\documentclass[a4paper]{article}
\usepackage[centertags, leqno]{amsmath}
\usepackage{amsthm}
\usepackage{amssymb}
\usepackage{amsfonts}
\usepackage{mathrsfs}
\usepackage{dsfont}
\usepackage{lineno}
\usepackage{geometry}
\usepackage{url}
\theoremstyle{plain}
\newtheorem{theorem}[equation]{Theorem}
\newtheorem{corollary}[equation]{Corollary}
\newtheorem{lemma}[equation]{Lemma}
\newtheorem{proposition}[equation]{Proposition}
\theoremstyle{definition}
\newtheorem{defi}[equation]{Definition}

\newtheorem{remark}[equation]{Remark}

\newcommand{\mydia}{\hfill $\Diamond$}

\newcommand{\assPro}[1]{% arg1: Symbol of the process
	\ensuremath{\langle #1 \rangle}}

\newcommand{\ind}{\text{\large{$\mathds{1}$}}}

\newcommand{\myd}{\mbox{\upshape d}}

\long\def\symbolfootnote[#1]#2{\begingroup%
\def\thefootnote{\fnsymbol{footnote}}\footnote[#1]{#2}\endgroup}

\numberwithin{equation}{section}

\title{One-dimensional Stochastic Differential Equations \\ with Generalized and Singular Drift\thanks{Work supported in part by the European Community's FP 7 Programme under contract PITN-GA-2008-213841, Marie Curie ITN "Controlled Systems".}}

\author{Stefan Blei \\ stefan.blei@uni-jena.de \and Hans-J\"urgen Engelbert \\ hans-juergen.engelbert@uni-jena.de}

\date{Friedrich-Schiller-Universit\"at Jena, \\ Fakult\"at f\"ur Mathematik und Informatik, \\ Institut f\"ur Stochastik, \\ D-07743 Jena, Germany \\[2ex] \today}

\begin{document}
\maketitle
\hrule 
\begin{abstract}
\noindent Introducing certain singularities, we generalize the class of one-dimensional stochastic differential equations with so-called generalized drift. Equations with generalized drift, well-known in the literature, possess a drift that is described by the semimartingale local time of the unknown process integrated with respect to a locally finite signed measure $\nu$. The generalization which we deal with can be interpreted as allowing more general set functions $\nu$, for example signed measures which are only $\sigma$-finite. However, we use a different approach to describe the singular drift. For the considered class of one-dimensional stochastic differential equations, we derive necessary and sufficient conditions for existence and uniqueness in law of solutions.
\\[2ex]
\end{abstract}
\hrule
\section{Introduction and Preliminaries}
\noindent Throughout this paper, $(\Omega, \mathcal{F}, \mathbf{P})$ stands for a complete probability space equipped with a filtration $\mathbb{F} = (\mathcal{F}_t)_{t \geq 0}$ which satisfies the usual conditions, i.e., $\mathbb{F}$ is right-continuous and $\mathcal{F}_0$ contains all sets from $\mathcal{F}$ which have $\mathbf{P}$-measure zero. For a process $X = (X_t)_{t \geq 0}$ the notation $(X,\mathbb{F})$ indicates that $X$ is $\mathbb{F}$-adapted. The processes considered in the following belong to the class of continuous semimartingales up to a stopping time $S$ and local times of such processes will play an important role. Therefore, in the Appendix we summarize some facts about continuous semimartingales up to a stopping time $S$ and their local times, which we use in the sequel. By writing (A.), we refer to a formula or a result of the Appendix.\\
\indent In the present paper, our purpose is to investigate one-dimensional stochastic differential equations (SDEs) with generalized and singular drift in the framework of continuous semimartingales. In general, this kind of SDEs admits exploding solutions. Therefore, the convenient state space is the extended real line $\overline{\mathbb{R}} = \mathbb{R}\cup \{-\infty, +\infty\}$ equipped with the $\sigma$-algebra $\mathscr{B}(\overline{\mathbb{R}})$ of Borel subsets. \\
\indent SDEs with \emph{generalized and singular drift} are of the form
\begin{equation}\label{eqn:SDE_mvasd}
	X_t = X_0 + \int_0^t b(X_s) \, \myd B_s + \int_\mathbb{R} L_m^X(t,y) \, \myd f(y)\,,
\end{equation}
where $b$ is a measurable real function and $B$ denotes a Wiener processes. Furthermore, the function $f$ which appears as the integrator in the drift of Eq. (\ref{eqn:SDE_mvasd}) is assumed to be non-negative, right-continuous and of locally bounded variation such that its reciprocal $1/f$ is locally integrable.\footnote{Here and in the following we use the convention $1/a = +\infty$ if $a=0$.} We call a function $f$ with these properties a \emph{drift function}. Moreover, $L_m^X$ denotes a certain local time of the unknown process $X$ specified in
\begin{defi}\label{def:solution}
	A continuous $(\overline{\mathbb{R}}, \mathscr{B}(\overline{\mathbb{R}}))$-valued stochastic process $(X,\mathbb{F})$ 
	defined on a probability space $(\Omega,\mathcal{F},\mathbf{P})$ is called a solution of Eq. (\ref{eqn:SDE_mvasd})
	if the following conditions are fulfilled: \medskip
	
	(i)   $X_0$ is real-valued.\medskip
	
  (ii)  $X_t = X_{t \wedge S_\infty^X}$, $t \geq 0$,	$\mathbf{P}$-a.s., 
								 where $S_\infty^X := \inf\{t \geq 0:|X_t| = +\infty\}$.\footnote{$\inf \emptyset = +\infty$.}\medskip
								 
	(iii) $(X,\mathbb{F})$ is a semimartingale up to $S_\infty^X$.\medskip
	
	(iv)  There exists a random function $L^X_m$ on $[0,S_\infty^X)\times \mathbb{R}$ into $[0,+\infty)$ that is a version of the local time of $X$ 
				defined on $\{t<S_\infty^X\}$ in the sense of an occupation times density with respect to the measure $m(\myd x) := 2\,f(x)\, \myd x$, i.e.,
				\[
					\int_0^t h(X_s) \, \myd \assPro{X}_s = \int_{\mathbb{R}} h(x) \, L_m^X(t,x)\,m(\myd x)\,, \qquad t < S_\infty^X, 
					\ \mathbf{P}\text{-a.s.}
				\]
				holds for all non-negative measurable functions $h$. Thereby, $L_m^X(\, . \, ,x)$ is $\mathbb{F}$-adapted for all $x \in \mathbb{R}$. 
				Moreover, $L_m^X$ is $\mathbf{P}$-a.s. continuous and increasing in $t$ as well as right-continuous in $x$ with limits from the left.\medskip
								 
	(v)	 There exists a Wiener process $(B,\mathbb{F})$ such that Eq. (\ref{eqn:SDE_mvasd}) is satisfied for all $t < S_\infty^X$ $\mathbf{P}$-a.s.
\end{defi}
\begin{defi}\label{def:uniqueness}
	We say that the solution of Eq. (\ref{eqn:SDE_mvasd}) (or of any other SDE appearing in the sequel) is unique 
	in law if any two solutions 
	$(X^1,\mathbb{F}^1)$ and $(X^2,\mathbb{F}^2)$ with coinciding initial distributions defined on the probability spaces 
	$(\Omega^1,\mathcal{F}^1,\mathbf{P}^1)$ and $(\Omega^2,\mathcal{F}^2,\mathbf{P}^2)$, respectively, possess the same image law on the space 
	$C_{\overline{\mathbb{R}}}([0,+\infty))$ of continuous functions defined on $[0,+\infty)$ and taking values in $\overline{\mathbb{R}}$.
\end{defi}
\indent The $\mathbb{F}$-stopping time $S^X_\infty$ in Definition \ref{def:solution} is called the \emph{explosion time} of $X$. To distinguish between $L_m^X$ and the (right) local time $L_+^X$, defined via (\ref{eqn:gen_ito_formula}), we also use the expression \emph{(right) semimartingale local time} when we refer to $L_+^X$. \\
\indent The idea to introduce such a local time $L_m^X$ in the context of SDEs for Dirichlet processes goes back to H.J. Engelbert and J. Wolf \cite{engelbertwolf}. Moreover, it was already used by S. Blei \cite{blei_bessel_2011} as a helpful tool in the investigation of an SDE for the $\delta$-dimensional Bessel process for $\delta \in (1,2)$, which turns out to be an important example of an equation of type (\ref{eqn:SDE_mvasd}) (see Section \ref{sec:skew_solutions}).\\
\indent By introducing certain singularities, the class of SDEs of the form (\ref{eqn:SDE_mvasd}) generalizes the class of SDEs with so-called generalized drift. SDEs with generalized drift have the following structure:
\begin{equation}\label{eqn:SDE_mvd}
	X_t = X_0 + \int_0^t b(X_s) \, \myd B_s + \int_\mathbb{R} L_+^X(t,y) \, \nu(\myd y)\,,
\end{equation}
where $b$ is a measurable real function and $\nu$ denotes a set function defined on the bounded Borel sets of the real line $\mathbb{R}$ such that it is a finite signed measure on $\mathscr{B}([-N,N])$ for every $N \in \mathbb{N}$. In this equation $B$ stands again for a Wiener processes and $L_+^X$ denotes the right semimartingale local time of the unknown process $X$. Omitting condition (iv), the notion of a solution of Eq. (\ref{eqn:SDE_mvd}) is introduced analogously to Definition \ref{def:solution}. Clearly, equations without drift
\begin{equation}\label{eqn:SDE_without_drift}
	Y_t = Y_0 + \int_0^t \sigma(Y_s) \, \myd B_s\,,
\end{equation}
where $\sigma$ is as well a measurable real function, are a special kind of Eq. (\ref{eqn:SDE_mvd}). \\
\indent SDEs of type (\ref{eqn:SDE_mvd}) with generalized drift have been studied previously by many authors. We refer the reader to Harrison and Shepp \cite{harrison_shepp}, N.I. Portenko \cite{portenko}, D.W. Stroock and M. Yor \cite{stroock_yor} and J.F. Le Gall \cite{LeGall_1983}, \cite{LeGall_1984}. H.J. Engelbert and W. Schmidt \cite{engelbert_schmidt:1985}, \cite{engelbert_schmidt:1989_III} derived rather weak necessary and sufficient conditions on existence and uniqueness of solutions to SDEs with generalized drift, which we recall in Theorem \ref{theorem:e_u_gen_drift} below. More recently, R.F. Bass and Z.-Q. Chen \cite{bass_chen} also considered SDEs of type (\ref{eqn:SDE_mvd}). \\
\indent We call the set function $\nu$ in Eq. (\ref{eqn:SDE_mvd}) \emph{drift measure} and we always assume additionally
\begin{equation}\label{eqn:cond_atoms}
	\nu(\{x\}) < \frac{1}{2}\,, \qquad x \in \mathbb{R}\,.
\end{equation}
Condition (\ref{eqn:cond_atoms}) is motivated by the fact that, in general, there is no solution of Eq. (\ref{eqn:SDE_mvd}) if $\nu(\{x\}) > 1/2$ for some $x \in \mathbb{R}$. The case $\nu(\{x\}) = 1/2$ for some $x \in \mathbb{R}$ corresponds to a reflecting barrier at the point $x$, which requires different methods to treat Eq. (\ref{eqn:SDE_mvd}) than by assuming (\ref{eqn:cond_atoms}) (cf. W. Schmidt \cite{schmidt:1989}, R.F. Bass and Z.-Q. Chen \cite{bass_chen}). In S. Blei and H.J. Engelbert \cite{blei_engelbert_2012} the reader can find a complete treatment of the features of Eq. (\ref{eqn:SDE_mvd}) in the cases $\nu(\{x\}) > 1/2$ and $\nu(\{x\}) = 1/2$ for some $x \in \mathbb{R}$. \\
\indent To see how equations of type (\ref{eqn:SDE_mvasd}) generalize equations of the form (\ref{eqn:SDE_mvd}), we consider the integral equation
\begin{equation}\label{eqn:integral_eqn}
	g_\nu(x) = \left\{\begin{array}{ll}
						1 - 2 \displaystyle\int_{[0,x]} g_\nu(y-) \, \nu(\myd y), & x \geq 0, \medskip \\
						1 + 2 \displaystyle\int_{(x,0)} g_\nu(y-) \, \nu(\myd y), & x < 0,
				 \end{array}\right.
\end{equation}
which has a unique c\`adl\`ag solution $g_\nu$. Note that for this statement condition (\ref{eqn:cond_atoms}) is needed. Moreover, $g_\nu$ is strictly positive as well as locally of bounded variation and the same properties also hold for the reciprocal function $1/g_\nu$. The explicit form of the solution $g_\nu$ can be found in \cite{engelbert_schmidt:1989_III}, (4.26). Setting $f_\nu := 1/g_\nu$, integration by parts gives
\begin{equation}\label{eqn:drift_measure}
	\nu(\myd y) =  \frac{1}{2} \, f^{-1}_\nu(y) \, \myd f_\nu(y)\, .
\end{equation}
Taking any solution $(X,\mathbb{F})$ of Eq. (\ref{eqn:SDE_mvd}), because of (\ref{eqn:drift_measure}), its drift can be expressed as
\[\begin{split}
 \int_\mathbb{R} L_+^X(t,y) \, \nu(\myd y) 
 			&= \int_\mathbb{R} \frac{1}{2} \, L_+^X(t,y) \, f^{-1}_\nu(y) \, \myd f_\nu(y)\,.
\end{split}\]
Therefore, defining 
\[
	L_m^X(t,y) := \frac{1}{2} \, L_+^X(t,y) \, f^{-1}_\nu(y)\,, \qquad (t,y) \in [0,S_\infty^X) \times \mathbb{R},
\]
we can write
\[
	X_t = X_0 + \int_0^t b(X_s) \, \myd B_s + \int_\mathbb{R} L_m^X(t,y) \, \myd {f_\nu}(y)\,.
\]
Furthermore, as an immediate consequence of the occupation times formula (\ref{eqn:occupationtime}) for $X$, we obtain that $L_m^X$ satisfies condition (iv) of Definition \ref{def:solution} with respect to the measure $m(\myd x) = 2 \, f_\nu(x) \, \myd x$. Altogether, $(X,\mathbb{F})$ is also a solution of Eq. (\ref{eqn:SDE_mvasd}) with diffusion coefficient $b$ and drift function $f_\nu$. \\
\indent Conversely, let us assume additionally that the reciprocal $1/f$ of the drift function $f$ in Eq. (\ref{eqn:SDE_mvasd}) is also of locally bounded variation or, equivalently, that $f$ and $f_-$ have no zeros.\footnote{For any real function $f$ with left hand limits $f(x-)$ we set $f_-(x) = f(x-)$, $x\in\mathbb{R}$.} Then, defining $\nu$ via (\ref{eqn:drift_measure}) by using $f$ on the right-hand side, we obtain a feasible drift measure $\nu$ fulfilling (\ref{eqn:cond_atoms}) and Eq. (\ref{eqn:SDE_mvasd}) reduces to an equation of type (\ref{eqn:SDE_mvd}) with $\nu$ as drift measure (see Remark \ref{remark:connection_mvd_mvasd}). Note that the corresponding function $f_\nu=1/g_\nu$, where $g_\nu$ is obtained via (\ref{eqn:integral_eqn}), differs from $f$ at most by a multiplicative constant. \\
\indent In contrast to Eq. (\ref{eqn:SDE_mvd}) regarded as an equation of type (\ref{eqn:SDE_mvasd}), one of the new features of Eq. (\ref{eqn:SDE_mvasd}) is that we do not postulate that $1/f$ is also of locally bounded variation. This is equivalent to allow $f$ and $f_-$ to have zeros. Responsible for the singularity of Eq. (\ref{eqn:SDE_mvasd}), these zeros play an important role in the following analysis of Eq. (\ref{eqn:SDE_mvasd}). We set $F_+ := \{x \in \mathbb{R} : f(x) = 0\}$ and $F_- := \{x \in \mathbb{R} : f(x-) = 0\}$. From our assumption that $1/f$ is locally integrable it follows immediately that the closed set $F:= F_+ \cup F_-$ is of Lebesgue-measure zero. Using an arbitrary drift function $f$ analogously as in (\ref{eqn:drift_measure}) to define $\nu(\myd y) :=  1/2 \, f^{-1}(y) \, \myd f(y)$,\footnote{We always use the convention $0 \cdot +\infty = 0$ and $a \cdot +\infty = +\infty$, $a \in \mathbb{R}\setminus\{0\}$.}
we obtain a set function $\nu$ which is in general not a finite signed measure on $\mathscr{B}([-N,N])$ for every $N \in \mathbb{N}$. Indeed, taking for example the drift function 
\[
	f(x) = \ind_{(-\infty,0)}(x) + \sqrt{x} \, \ind_{(0, +\infty)}(x), \qquad x\in\mathbb{R},
\]
for any $N \in \mathbb{N}$, the corresponding $\nu$ is a signed measure on $\mathscr{B}([-N,N])$ which is only $\sigma$-finite. But as seen from the drift function $f(x) = \sqrt{|x|}$, $x\in\mathbb{R}$, it is also possible that $\nu$ is no longer a signed measure on $\mathscr{B}([-N,N])$ because for some sets it takes the value $+\infty$ and for other sets the value $-\infty$. \\
\indent Besides these examples of a singleton $F$, it is of course possible that $F$ is even an uncountable (e.g., Cantor like) set. Nevertheless, already in the case that $F$ consists only of one point, surprising and interesting effects like skewness and reflection can be observed for solutions of Eq. (\ref{eqn:SDE_mvasd}). In the context of the complexity of the set $F$, it becomes clear quite quickly that, in general, solutions of Eq. (\ref{eqn:SDE_mvasd}) go beyond the scope of semimartingales. But it is the objective of this paper to approach the problem in a first step keeping within the framework of semimartingales. \\
\indent The paper is organized as follows. We begin with stating some useful properties of the local time $L_m^X$ of a solution $(X,\mathbb{F})$ of Eq. (\ref{eqn:SDE_mvasd}) in Section \ref{sec:prop_loc_time}. Afterwards in Section \ref{sec:space_transform} we investigate the structure of a solution $(X,\mathbb{F})$ of Eq. (\ref{eqn:SDE_mvasd}). In particular, we show the connection of Eq. (\ref{eqn:SDE_mvasd}) to another equation arising from a space transformation and prepare the study of existence and uniqueness in law of solutions of Eq. (\ref{eqn:SDE_mvasd}), which is done in Section \ref{sec:symmetric_solutions} and \ref{sec:skew_solutions}. \\
\indent As preliminaries in the context of existence and uniqueness of solutions we briefly recall some known facts for the subclass of equations of type (\ref{eqn:SDE_mvd}), which we will use later. For any real measurable function $h$, we introduce the sets
\[
	N_h := \{x \in \mathbb{R} : h(x) = 0 \}
\]
and
\[
	E_h := \{x \in \mathbb{R}: \int_U h^{-2}(y) \, \myd y = + \infty \text{ for all open sets } U \text{ containing } x\}\,.
\]
The following property of a solution of Eq. (\ref{eqn:SDE_mvd}) is proven in \cite{engelbert_schmidt:1989_III}, Proposition (4.34). For the special case of an equation of type (\ref{eqn:SDE_without_drift}) without drift, the corresponding statement can also be found in \cite{engelbert_schmidt:1989_III}, Proposition (4.14).
\begin{lemma}\label{lemma:stopping}
	Let $(X,\mathbb{F})$ be a solution of Eq. (\ref{eqn:SDE_mvd}). Then we have $X_t = X_{t \wedge D_{E_b}^X}$, $t \geq 0$, 
	$\mathbf{P}$-a.s., where $D_{E_b}^X$ denotes the first entry time of $X$ into the set $E_b$.
\end{lemma}

\indent The next result gives conditions on existence and uniqueness of solutions of Eq. (\ref{eqn:SDE_mvd}). For the proofs we refer to H.J. Engelbert and W. Schmidt \cite{engelbert_schmidt:1989_III}, Theorem (4.35) and (4.37) (see also \cite{engelbert_schmidt:1985}, Theorem 3 and 4). These statements were established by reducing Eq. (\ref{eqn:SDE_mvd}) to an equation of type (\ref{eqn:SDE_without_drift}) without drift. Hence, analogous statements for equations of type (\ref{eqn:SDE_without_drift}), which are verified in \cite{engelbert_schmidt:1989_III}, Theorem (4.17) and (4.22) (see also \cite{engelbert_schmidt:1985}, Theorem 1 and 2), could be used to derive the results for Eq. (\ref{eqn:SDE_mvd}).
\begin{theorem}\label{theorem:e_u_gen_drift}
	(i) For every initial distribution there exists a solution of Eq. (\ref{eqn:SDE_mvd}) if and only if the condition $E_b \subseteq N_b$ 
	for the diffusion coefficient $b$ is satisfied. \medskip
	
	(ii) For every initial distribution there exists a unique solution of Eq. (\ref{eqn:SDE_mvd}) if and only if the condition $E_b = N_b$ 
	for the diffusion coefficient $b$ is satisfied.
\end{theorem}

\section{Properties of the Local Time $\text{\textit{L}}_{\text{\textit{\lowercase{m}}}}^{\text{\textit{X}}}$}\label{sec:prop_loc_time}
\noindent In this part of the paper, for a solution $(X,\mathbb{F})$ of Eq. (\ref{eqn:SDE_mvasd}) we investigate the relation between the local time $L_m^X$ and the semimartingale local time $L_+^X$. In particular, we will see that $L_m^X$ inherits useful, well-known properties of $L_+^X$.
\begin{lemma}\label{lemma:connection_loc_times}
	Let $(X,\mathbb{F})$ be a solution of Eq. (\ref{eqn:SDE_mvasd}). Then we have
	\[
		 2\,f(y\pm)\,L^X_m(t,y\pm) = L_\pm^X(t,y), \qquad t < S_\infty^X,\ y\in \mathbb{R}, \ \mathbf{P}\text{-a.s.}
		 \footnote{Note that $L_m^X(t,y+) = L^X_m(t,y),	\ t < S_\infty^X, \ y\in\mathbb{R}, \ \mathbf{P}\text{-a.s.}$}
	\]
\end{lemma}
\begin{proof} Due to the occupation times formula (\ref{eqn:occupationtime}) and Definition \ref{def:solution}(iv) it holds
\[
 \int_{\mathbb{R}} h(y) \, L_m^X(t,y) \, 2 \, f(y)\,\myd y = \int_{\mathbb{R}} h(y) \, L_+^X(t,y) \, \myd y, \qquad t< S_\infty^X, \ \mathbf{P}\text{-a.s.}
\]
for every measurable function $h \geq 0$. This implies
\[
	2 \, f(y) \, L_m^X(t,y) = L_+^X(t,y) \qquad \lambda\text{-a.e.}, \ t<S_\infty^X, \ \mathbf{P}\text{-a.s.}
\]
Using the continuity properties of the involved objects, we obtain the assertions of the lemma.
\end{proof}
\begin{remark}\label{remark:connection_mvd_mvasd}
	By means of Lemma \ref{lemma:connection_loc_times} we can conclude: If the reciprocal $1/f$ 
	of the drift function $f$ is of locally bounded variation or, equivalently, if $F = \emptyset$, then Eq. (\ref{eqn:SDE_mvasd}) reduces to an 
	equation of type (\ref{eqn:SDE_mvd}). Indeed, using $\nu(\myd y) := 1/2 \, f^{-1}(y) \, \myd f(y)$ to define a feasible drift measure $\nu$ which 
	satisfies (\ref{eqn:cond_atoms}), it follows
  \[
	 \int_\mathbb{R} L_m^X(t,y)\, \myd f(y) 
	 		= \int_\mathbb{R} L_+^X(t,y) \, \frac{1}{2 \, f(y)} \, \myd f(y) 
	 		= \int_\mathbb{R} L_+^X(t,y)\, \nu(\myd y)
  \]
  for the drift part in Eq. (\ref{eqn:SDE_mvasd}). \mydia
\end{remark}

From Lemma \ref{lemma:connection_loc_times} we obtain the following corollaries. The first one is obvious. 
\begin{corollary}\label{corr:beingzero}
	For a solution $(X,\mathbb{F})$ of Eq. (\ref{eqn:SDE_mvasd}) it holds
	\[
		L_\pm^X(t,y)= 0, \qquad t < S_\infty^X, \ y \in F_\pm, \ \mathbf{P}\text{-a.s.}
	\]
\end{corollary}
Corollary \ref{corr:beingzero} shows that we always have $L_+^X(t,y)=0$ if $y \in F_+$ and $L_-^X(t,y)=0$ if $y \in F_-$. For $L_m^X$, in general, this does not hold. In contrast, as seen from Lemma \ref{lemma:connection_loc_times}, $L_m^X$ gives a precise description of the asymptotic behaviour of $L_+^X(t,y) \, f^{-1}(y)$ in the zeros of $f$ and $f_-$.
\begin{corollary}\label{corr:comp_interval}
	The local time $L_m^X$ of a solution $(X,\mathbb{F})$ of Eq. (\ref{eqn:SDE_mvasd}) satisfies
	\[
		L_m^X(t,y\pm) = 0 \text{ on } \left\{y \notin \left[\min_{0 \leq s \leq t} X_s, \max_{0 \leq s \leq t} X_s\right]\right\}, 
										\ t < S_\infty^X, \ \mathbf{P}\text{-a.s.}
	\]
\end{corollary}
\begin{proof} It is well-known that for the semimartingale local times we have 
\[
	L_\pm^X(t,y) = 0 \text{ on } \left\{y \notin \left[\min_{0 \leq s \leq t} X_s, \max_{0 \leq s \leq t} X_s\right]\right\}, 
										\ t < S_\infty^X, \ \mathbf{P}\text{-a.s.}
\]
From Lemma \ref{lemma:connection_loc_times} we see immediately
\[
	L_m^X(t,y \pm) = 0 \text{ on } \left\{y \notin \left[\min_{0 \leq s \leq t} X_s, \max_{0 \leq s \leq t} X_s\right] \cup F \right\}, 
										 \ t < S_\infty^X, \ \mathbf{P}\text{-a.s.}
\]
Since $F$ has Lebesgue measure zero, the continuity properties of $L_m^X$ imply the desired result.
\end{proof}

Based on Corollary \ref{corr:comp_interval}, we use the convention $L_m^X(t,\pm\infty ) := 0$, $t < S_\infty^X$. \\
\indent Next we prove that, similar to property (\ref{eqn:int_wrt_loc_time}) of the semimartingale local times, the measures $L_m^X(\myd t , y)$ and $L_m^X(\myd t, y-)$ do not charge the set $\{0 \leq t < S_\infty^X: X_t \neq y\}$.
\begin{theorem}\label{theorem:supp_loc_time}
	Let $(X,\mathbb{F})$ be a solution of Eq. (\ref{eqn:SDE_mvasd}). Then, for every $y \in \mathbb{R}$, $L_m^X(\myd t,y\pm)$ is carried by 
	$\{0 \leq t < S_\infty^X : X_t = y\}$ $\mathbf{P}$-a.s., i.e.,
	\[
		\int_0^t \ind_{\{y\}}(X_s) \, L_m^X(\myd s, y\pm) = L_m^X(t,y\pm), \qquad t < S_\infty^X, \ y \in \mathbb{R}, \ \mathbf{P}\text{-a.s.}
	\]
\end{theorem}

\begin{proof} We prove the statement for $L_m^X$. The verification for the left-continuous version of $L_m^X$ can be done in an analogous way. Using property (\ref{eqn:int_wrt_loc_time}) of the semimartingale local time $L_+^X$ and Lemma \ref{lemma:connection_loc_times}, we conclude
\begin{equation}\label{eqn:first_part}
	\begin{split}
		\int_0^t \ind_{\{y\}}(X_s) \, L_m^X(\myd s, y) &= \frac{1}{2f(y)} \, \int_0^t \ind_{\{y\}}(X_s) \, L_+^X(\myd s, y)\\
																									 &= \frac{1}{2f(y)} \, L_+^X(t, y) \\
																									 &= L_m^X(t, y), \qquad t < S_\infty^X, \ y \in \mathbb{R}\setminus F_+, \ \mathbf{P}\text{-a.s.}
	\end{split}
\end{equation}
Therefore, it remains to show that the asserted relation is also fulfilled for the points of the set $F_+$. Let $\Omega_0 \in \mathcal{F}$ be such that $\mathbf{P}(\Omega_0) = 1$, $X_0(\omega) \in \mathbb{R}$, $\omega \in \Omega_0$, and with the property that, for all $\omega \in \Omega_0$, $L_m^X(\,.\,,\,.\,)(\omega)$ satisfies \ref{def:solution}(iv) as well as (\ref{eqn:first_part}). Given $\omega \in \Omega_0$ and $a \in F_+$, we can, because $F$ has Lebesgue measure zero, choose a sequence $(y_n)_{n \in \mathbb{N}} \subseteq \mathbb{R}\setminus F_+$ with $y_{n+1} \leq y_n$ and $\lim_{n \rightarrow +\infty} y_n = a$. Moreover, let $(S_i)_{i \in \mathbb{N}} \subseteq [0,+\infty)$ be an increasing sequence of real numbers with $S_i < S_\infty^X(\omega)$, $i \in \mathbb{N}$, and $\lim_{i \rightarrow +\infty} S_i = S_\infty^X(\omega)$. Then, for fixed $i \in \mathbb{N}$ 
\[
	\int_B L_m^X(\myd s,a)(\omega) \quad \text{and} \quad \int_B L_m^X(\myd s,y_n)(\omega), \ n \in \mathbb{N}, \qquad B \in 
	\mathscr{B}([0,S_i]),
\]
are finite measures on $([0,S_i],\mathscr{B}([0,S_i]))$. Since $L_m^X(t,\, . \,)(\omega)$ is right-continuous in $a$ for every $t \in [0,S_i]$, the sequence of measures
\[
	\left(\int_{\text{.}} L_m^X(\myd s, y_n)(\omega)\right)_{n \in \mathbb{N}}
\]
converges weakly to the measure $\int_{\text{.}} L_m^X(\myd s, a)(\omega)$. From the continuity of $X_{.}(\omega)$, we obtain that the sets
\[
	T_k := [0,S_i] \cap (X_{.}(\omega))^{-1}\textstyle\left((-\infty,a-\frac{1}{k})\cup (a+\frac{1}{k},+\infty)\right), \qquad k \in \mathbb{N},
\]
are open in $[0,S_i]$. By the weak convergence of the considered measures, for every $k \in \mathbb{N}$ it follows 
\[
	\int_{T_k} L_m^X(\myd s, a)(\omega) \leq \liminf\limits_{n \rightarrow +\infty} \int_{T_k} L_m^X(\myd s, y_n)(\omega)\,.
\]
Additionally, for every $k \in \mathbb{N}$ there exists an $n_0(k) \in \mathbb{N}$ such that $y_n \in (a-\frac{1}{k},a+\frac{1}{k})$, $n \geq n_0(k)$, and together with (\ref{eqn:first_part}) we conclude
\[
	\int_{T_k} L_m^X(\myd s, y_n)(\omega) = 0, \qquad n \geq n_0(k),
\]
and therefore 
\[
	\int_{T_k} L_m^X(\myd s, a)(\omega) = 0, \qquad k \in \mathbb{N}.
\]
Via the continuity from below of the considered measure we get
\[\begin{split}
	\int_{\{t \in [0,S_i]: X_t(\omega) \neq a\}} L_m^X(\myd s, a)(\omega) 
						& = \int_{\bigcup\limits_{k=1}^{+\infty} T_k} L_m^X(\myd s, a)(\omega) \\
						& = \lim_{k \rightarrow +\infty} \int_{T_k} L_m^X(\myd s, a)(\omega) \\
						& = 0.
\end{split}\]
Considering the measure on $[0,S_\infty^X(\omega))$ and applying again the continuity from below, we finally observe
\[\begin{split}
	\int_{\{t \in [0,S_\infty^X(\omega)): X_t(\omega) \neq a\}} L_m^X(\myd s, a)(\omega) 
						& = \int_{\bigcup\limits_{i=1}^{+\infty} \{t \in [0,S_i]: X_t(\omega) \neq a\}} L_m^X(\myd s, a)(\omega) \\
						& = \lim_{i \rightarrow +\infty} \int_{\{t \in [0,S_i]: X_t(\omega) \neq a\}} L_m^X(\myd s, a)(\omega) \\
						& = 0.
\end{split}\]
Since $\omega \in \Omega_0$ and $a \in F_+$ were chosen arbitrarily, the proof is finished. 
\end{proof}

The last theorem allows us to conclude that the local time $L_m^X$ is also $\mathbf{P}$-a.s. continuous in the state variable except in the points of $F_-$.
\begin{corollary}\label{corollary:left_continuity}
	Let $(X,\mathbb{F})$ be a solution of Eq. (\ref{eqn:SDE_mvasd}). Then it holds
	\[
		L_m^X(t,y) = L_m^X(t,y-), \qquad t < S_\infty^X ,\ y \in \mathbb{R}\setminus F_-, \ \mathbf{P}\text{-a.s.}
	\]
\end{corollary}
\begin{proof} Using property (\ref{eqn:loc_time_and_variation_process}) and Theorem \ref{theorem:supp_loc_time}, we conclude
\[
	\begin{split}
		L_+^X(t,y) - L_-^X(t,y) &= 2 \int_0^t \ind_{\{y\}}(X_s) \, \int_{\mathbb{R}} L_m^X(\myd s,z) \, \myd f(z) \\
													&= 2 \int_{\mathbb{R}} \int_0^t \ind_{\{y\}}(X_s) \,  L_m^X(\myd s,z) \, \myd f(z) \\ 
													&= 2 \int_{\{y\}} L_m^X(t,z) \, \myd f(z) \\ 
													&= 2 \,L_m^X(t,y) \,(f(y) - f(y-)), \qquad t < S_\infty^X,\ y \in \mathbb{R},\ \mathbf{P}\text{-a.s.} 
	\end{split}
\]
Together with Lemma \ref{lemma:connection_loc_times} this yields
\[
	\begin{split}
		L_+^X(t,y) - L_-^X(t,y) &= L_+^X(t,y) - 2 \,L_m^X(t,y) \, f(y-), \qquad t < S_\infty^X,\ y \in \mathbb{R}, \ \mathbf{P}\text{-a.s.}, 
	\end{split}
\]
and finally
\[
	L_m^X(t,y) = \frac{1}{2f(y-)} \, L_-^X(t,y) = L_m^X(t,y-), \qquad t < S_\infty^X,\ y \in \mathbb{R}\setminus F_-, \ \mathbf{P}\text{-a.s.},
\]
the desired result. 
\end{proof}

The last lemma which we give in this section is used below to infer the existence of certain stochastic integrals.
\begin{lemma}\label{lemma:integrability}
Let $(X,\mathbb{F})$ be a solution of Eq. (\ref{eqn:SDE_mvasd}). Then
\[
	\int_{\mathbb{R}} \frac{1}{f(y)} \, L_m^X(t,y) \, \myd y < +\infty, \qquad t 	< S_\infty^X, \ \mathbf{P}\text{-a.s.}
\]
and
\[
	\int_0^t \left(\frac{1}{f}(X_s)\right)^2 \, \myd \assPro{X}_s < +\infty, \qquad t 	< S_\infty^X, \ \mathbf{P}\text{-a.s.}
\]
hold true.
\end{lemma}
\begin{proof} The occupation times formula for $L_m^X$ implies
\[\begin{split}
	\int_0^t \left(\frac{1}{f}(X_s)\right)^2 \, \myd \assPro{X}_s 
					&= \int_{\mathbb{R}} \left(\frac{1}{f(y)}\right)^2 \, L_m^X(t,y) \, 2 \, f(y) \, \myd y \\
					&= 2\int_{\mathbb{R}} \frac{1}{f(y)} \, L_m^X(t,y) \, \myd y, \qquad t < S_\infty^X, \ \mathbf{P}\text{-a.s.}
\end{split}\]
Therefore, it is enough to show the first statement. Because of Corollary \ref{corr:comp_interval} we have
\[
	\int_{\mathbb{R}} \frac{1}{f(y)} \, L_m^X(t,y) \, \myd y
	 = \int_{\left[\min\limits_{0\leq s \leq t} X_s, 
	   \max\limits_{0\leq s \leq t} X_s\right]} \frac{1}{f(y)} \, L_m^X(t,y) \, \myd y,
	   \qquad t < S_\infty^X, \ \mathbf{P}\text{-a.s.}
\]
Moreover, for $\mathbf{P}$-a.e. $\omega\in \Omega$ and $t < S_\infty^X(\omega)$ we can find a constant $K_t(\omega) > 0$ such that
$L_m^X(t, \, . \,)(\omega)$ is bounded by $K_t(\omega)$ on $\left[\min\limits_{0\leq s \leq t} X_s(\omega), \max\limits_{0\leq s \leq t} X_s(\omega)\right]$. For $\mathbf{P}$-a.e. $\omega \in \Omega$ and $t < S_\infty^X(\omega)$ it follows
\[\begin{split}
	\int_{\mathbb{R}} \frac{1}{f(y)} \, L_m^X(t,y)(\omega) \, \myd y 
		& \leq K_t(\omega) \int_{\left[\min\limits_{0\leq s \leq t}X_s(\omega), \max\limits_{0\leq s \leq t} X_s(\omega)\right]} \frac{1}{f(y)}
			\, \myd y 
		< +\infty,
\end{split}\]
where we used the fact that $1/f$ is locally integrable.
\end{proof}

\section{Space Transformation}\label{sec:space_transform}
\noindent As pointed out in the introduction equations of type (\ref{eqn:SDE_mvd}) with generalized drift are contained as a special case in the class of equations (\ref{eqn:SDE_mvasd}) with generalized and singular drift. Using a certain space transformation, the so-called Zvonkin transformation (see \cite{zvonkin}), it is well-known that Eq. (\ref{eqn:SDE_mvd}) can be reduced to an equation (\ref{eqn:SDE_without_drift}) without drift. Based on the generalized It\^o formula, this method has been used by many authors (see e.g. \cite{bass_chen}, \cite{engelbert_schmidt:1985}, \cite{engelbert_schmidt:1989_III}, \cite{LeGall_1984}, \cite{stroock_yor}) to study Eq. (\ref{eqn:SDE_mvd}). They were able to derive conditions on existence and uniqueness of solutions of Eq. (\ref{eqn:SDE_mvd}) from the well-known criteria on existence and uniqueness of solutions of Eq. (\ref{eqn:SDE_without_drift}). \\
\indent For the treatment of Eq. (\ref{eqn:SDE_mvasd}) we want to use a similar approach. A natural candidate for an appropriate transformation of Eq. (\ref{eqn:SDE_mvasd}) is the strictly increasing and continuous primitive
\[
	G(x) := \int_0^x \frac{1}{f(y)} \, \myd y, \qquad x \in \overline{\mathbb{R}},
\]
of the locally integrable reciprocal $1/f$ of the drift function $f$.
By $H$ we denote the inverse of $G$ given on $G(\mathbb{R})=(G(-\infty),G(+\infty))$. We extend the functions $H$, $f$, $1/f$ and $b$ by setting
\[
	H(x) := \left\{\begin{array}{ll}
							+ \infty, & x \in [G(+\infty),+\infty], \bigskip \\ 
							- \infty, & x \in [-\infty,G(-\infty)],
				       \end{array}\right.
\]
\[
	\frac{1}{f(\pm\infty)} = b(\pm\infty) := 0 \qquad \text{and} \qquad f(\pm \infty) := +\infty\,.
\]
Clearly, $H$ satisfies
\begin{equation}\label{eqn:representation_H}
	H(x) = \int_0^x (f \circ H)(y) \, \myd y, \qquad x \in \mathbb{R}\,.
\end{equation}
The open set $\mathbb{R}\setminus F$ can be uniquely decomposed into at most countably many open intervals, i.e.,
\begin{equation}\label{eqn:components}
	\mathbb{R}\setminus F = \bigcup_{i=0}^{|F|} \, (a_i,b_i)\,, \quad \text{where } a_i, b_i \in F \cup \{-\infty,+\infty\}
\end{equation}
and\footnote{$\mathbb{N}_0 = \{0,1,2,\ldots\}$} $|F| \in \mathbb{N}_0 \cup \{+\infty\}$ denotes the number of elements in $F$. Note that $|F| + 1$ is just the number of intervals $(a_i,b_i)$ in the representation (\ref{eqn:components}). Already in Corollary \ref{corollary:left_continuity} it turned out that the subset $F_-$ of $F$ takes up a special role. We remark that $F$ is countable if and only if $F_-$ is countable. Indeed, from (\ref{eqn:components}) the reader can see
\[
  \left\{x \in \mathbb{R}: \, (x,x+\varepsilon) \cap F\neq \emptyset, \, (x-\varepsilon,x) \cap F \neq \emptyset \ \forall\,\varepsilon > 0\right\} 		= F\setminus \bigcup_{i=0}^{|F|} \{a_i, b_i\}  \subseteq F_- \subseteq F,
\]
where on the left-hand side stands the set of all points which are accumulation points from the left and from the right in $F$. Hence, we can conclude. \\
\indent For a continuous $(\overline{\mathbb{R}}, \mathscr{B}(\overline{\mathbb{R}}))$-valued process $(Y,\mathbb{F})$ we introduce the $\mathbb{F}$-stopping time 
\[
	S^Y_{G(\mathbb{R})} := \inf\{t \geq 0: Y_t = G(+\infty) \text{ or } Y_t = G(-\infty)\}\,,
\]
Now, in full generality, we can give the following structure of a transformed solution $Y = G(X)$.
\begin{theorem}\label{theorem:general_structure_space_transformation}
	Let $(X,\mathbb{F})$ be a solution of Eq. (\ref{eqn:SDE_mvasd}) with generalized and singular drift. Then the process $(Y,\mathbb{F})$ defined by 
	$Y =G(X)$ is continuous with values in $(\overline{\mathbb{R}}, \mathscr{B}(\overline{\mathbb{R}}))$ and stopped when it leaves the open
	interval $G(\mathbb{R})$, i.e., $Y_t = Y_{t \wedge S^Y_{G(\mathbb{R})}}$, $t \geq 0$, $\mathbf{P}$-a.s. Moreover, setting $\sigma := (b/f)\circ H$,
	then $Y$ satisfies
	\[
		Y_t = Y_0 + \int_0^t \sigma(Y_s) \, \myd B_s + \sum_{i = 0}^{|F|} \left(L_m^X(t,a_i) - L_m^X(t,b_i-) \right), 
					\qquad t < S^Y_{G(\mathbb{R})}, \ \mathbf{P}\text{-a.s.},
	\]
	where in case $|F| = +\infty$ the process
	\[
		\left(\left(\displaystyle\sum_{i = 0}^{+\infty} \left(L_m^X(t,a_i) - L_m^X(t,b_i-) \right)\right)_{t \geq 0},\, 
		\mathbb{F}\right)
	\]
	fulfils the relation
	\[
		\lim_{n \rightarrow +\infty} \sup_{0 \leq t \leq T} \left| \sum_{i=0}^n \left( L_m^X(t,a_i) - L_m^X(t,b_i) \right) 
		- \sum_{i = 0}^{+\infty} \left(L_m^X(t,a_i) - L_m^X(t,b_i-) \right)\right| = 0
	\]
	in probability on $\{T < S^Y_{G(\mathbb{R})}\}$ for all $T \geq 0$.
\end{theorem}
\begin{remark}\label{remark:diffcoeff_infty}
	We point out that the statement of Theorem \ref{theorem:general_structure_space_transformation} includes the existence of the 
	stochastic integral appearing in the decomposition of $Y$, i.e., the property
	\[
		\int_0^t \sigma^2(Y_s) \, \myd s < +\infty, \qquad t < S^Y_{G(\mathbb{R})}, \ \mathbf{P}\text{-a.s.}
	\]
	In particular, because $|\sigma(x)| = +\infty$, $x \in G(F_+ \cap \{b \neq 0\})$,\footnote{For $A \subseteq \mathbb{R}$, $G(A)$ denotes the image of 
	the set $A$ under $G$.} from this follows that
	\begin{equation}\label{eqn:no_occ_time}
		\int_0^t \ind_{G(F_+ \cap \{b \neq 0\})} (Y_s) \, \myd s = 0, \qquad t < S^Y_{G(\mathbb{R})}, \ \mathbf{P}\text{-a.s.},
	\end{equation}
	i.e., $Y$ has no occupation time in $G(F_+ \cap \{b \neq 0\})$ $\mathbf{P}$-a.s. This can also be derived immediately from the fact that any 
	solution $(X,\mathbb{F})$ of Eq. (\ref{eqn:SDE_mvasd}) has no occupation time in $F_+ \cap \{b \neq 0\}$:
	\[\begin{split}
		\int_0^t \ind_{F_+ \cap \{b \neq 0\}} (X_s) \, \myd s 
				& = \int_0^t \ind_{F_+ \cap \{b \neq 0\}} (X_s) \, b^{-2}(X_s) \, \myd \assPro{X}_s \\
				& = \int_{F_+ \cap \{b \neq 0\}} \, L_+^X(t,x) \, b^{-2}(x) \, \myd x \\
				& = 0, \qquad t < S_\infty^X, \ \mathbf{P}\text{-a.s.}\,,
	\end{split}\]
	where we used the occupation times formula (\ref{eqn:occupationtime}) and the fact that $F_+$ is of Lebesgue measure zero. Clearly, this also 
	implies (\ref{eqn:no_occ_time}). It might be uncommon to consider SDEs with diffusion coefficient $\sigma$ taking also infinite values. However, 
	because of (\ref{eqn:no_occ_time}) we can alter $\sigma$ on $G(F_+ \cap \{b \neq 0\})$ by setting, e.g.,
	\[
		\widetilde{\sigma} = \sigma \, \ind_{\mathbb{R}\setminus G(F_+ \cap \{b \neq 0\})} + \ind_{G(F_+ \cap \{b \neq 0\})}
	\]
	and replace $\sigma$ by $\widetilde{\sigma}$ without changing the statement of Theorem \ref{theorem:general_structure_space_transformation}. 
	In the following we still use 
	$\sigma = (b/f)\circ H$ but keep in mind that, always when it is necessary, we can use an appropriate real-valued coefficient. In this context 
	we also refer to Remark \ref{remark:occupation_time_in_zero}. \mydia
\end{remark}
\begin{proof}[Proof of Theorem \ref{theorem:general_structure_space_transformation}]
Since $(X,\mathbb{F})$ is a solution of Eq. (\ref{eqn:SDE_mvasd}), the process $(Y,\mathbb{F})$ is obviously continuous with values in $(\overline{\mathbb{R}}, \mathscr{B}(\overline{\mathbb{R}}))$ such that $Y_0 \in G(\mathbb{R})$ $\mathbf{P}$-a.s. Clearly,
\[
	S_\infty^X = \inf\{t \geq 0: X_t \notin \mathbb{R}\} = \inf\{t \geq 0: Y_t \notin G(\mathbb{R})\} = S^Y_{G(\mathbb{R})}
\]
and from Definition \ref{def:solution}(ii) we obtain immediately 
\[
	Y_t = G(X_t) 
			= G(X_{t \wedge S_\infty^X}) 
			= G(X_{t \wedge S^Y_{G(\mathbb{R})}}) 
			= Y_{t \wedge S^Y_{G(\mathbb{R})}}, \qquad t \geq 0, \ \mathbf{P}\text{-a.s.}
\]
To show the claimed structure of $Y$, one is tempted to apply the generalized It\^o formula (\ref{eqn:gen_ito_formula}), but $G$ restricted to $\mathbb{R}$ is, in general, not the difference of convex functions which is caused by the potentially non-empty set $F$. To overcome this difficulty, we use the decomposition (\ref{eqn:components}) of the open set $\mathbb{R}\setminus F$. For every  $i \in \{0, \ldots, |F|\} \cap \mathbb{N}_0$ we choose two sequences $(p_k^i)_{k \in \mathbb{N}}$ and $(q_k^i)_{k \in \mathbb{N}}$ such that
\[
	a_i < p_{k+1}^i < p_k^i < q_k^i < q_{k+1}^i < b_{i}, \ k \in\mathbb{N}, \qquad \lim_{k \rightarrow +\infty} p_k^i = a_i 
	\qquad \text{and} \qquad \lim_{k \rightarrow +\infty} q_k^{i} = b_{i}
\]
are satisfied and define
\[
	g_{n,k} := \frac{1}{f} \, \ind_{\bigcup\limits_{i=0}^n \left[p_k^i, q_k^i \right)} 
					= \sum\limits_{i=0}^n \frac{1}{f} \, \ind_{\left[p_k^i\, , \, q_k^i\right)}, \qquad k \in \mathbb{N}, \ n \in  \{0, \ldots, |F|\} \cap 
					\mathbb{N}_0\,,
\]
as well as
\[
	g_n := \frac{1}{f} \, \ind_{\bigcup\limits_{i=0}^n \left(a_i, b_i \right)}, \qquad n \in  \{0, \ldots, |F|\} \cap \mathbb{N}_0\,.
\]
Moreover, we introduce the increasing and continuous functions
\[
	G_{n,k}(x) := \int_0^x g_{n,k} (y) \, \myd y, \qquad x \in \overline{\mathbb{R}}, \ k \in \mathbb{N}, \ n \in \{0, \ldots, |F|\} \cap \mathbb{N}_0\,,
\]
and
\[
	G_n(x) := \int_0^x g_n (y) \, \myd y, \qquad x \in \overline{\mathbb{R}}, \ n \in \{0, \ldots, |F|\} \cap \mathbb{N}_0\,.
\]
Clearly, the non-negative right-continuous functions $g_{n,k}$, $k \in \mathbb{N}$, $n \in  \{0, \ldots, |F|\} \cap \mathbb{N}_0$, are of locally bounded variation, since every summand in the finite sum of the definition of $g_{n,k}$ is of locally bounded variation. Therefore, for arbitrary $n,k \in \mathbb{N}$ we conclude that $G_{n,k}$ restricted to $\mathbb{R}$ is the difference of convex functions. Applying the generalized It\^o formula (\ref{eqn:gen_ito_formula}), we deduce
\begin{equation}\label{eqn:zwischenrechnung78}
	\begin{split}
		G_{n,k}(X_t)&= G_{n,k}(X_0) + \int_0^t g_{n,k}(X_s -) \, \myd X_s + \frac{1}{2} \int_\mathbb{R} L_+^X(t,y) \, \myd g_{n,k}(y) \\
								&= G_{n,k}(X_0) + \int_0^t g_{n,k}(X_s)\, b(X_s) \, \myd B_s + \int_0^t g_{n,k}(X_s-) \int_\mathbb{R} L_m^X(\myd s,y)\,\myd f(y) \\
								&\phantom{=======} + \frac{1}{2} \int_\mathbb{R} L_+^X(t,y) \, \myd g_{n,k}(y), 
	\end{split}
\end{equation}
$t < S^Y_{G(\mathbb{R})}, \ k \in \mathbb{N}, \ n \in \{0,\ldots,|F|\} \cap \mathbb{N}_0, \ \mathbf{P}\text{-a.s.}$
For the third summand in this decomposition we obtain
\[\begin{split}
	\int_0^t g_{n,k}(X_s-) \, \int_{\mathbb{R}} L_m^X(\myd s,y) \, \myd f(y)
		&= \int_{\mathbb{R}} \int_0^t g_{n,k}(X_s-) \, L_m^X(\myd s,y) \, \myd f(y) \\
		&= \int_{\mathbb{R}} g_{n,k}(y-) \, L_m^X(t,y) \, \myd f(y),
\end{split}\]
$t < S^Y_{G(\mathbb{R})}, \ k \in \mathbb{N}, \ n \in \{0,\ldots,|F|\} \cap \mathbb{N}_0, \ \mathbf{P}\text{-a.s.}$, where we used Theorem \ref{theorem:supp_loc_time}. Moreover, integration by parts gives
\[\begin{split}
	\int_{\mathbb{R}} g_{n,k}(y-) \, L_m^X(t,y) \, \myd f(y)
		&= \int_{\mathbb{R}} L_m^X(t,y) \, \myd (g_{n,k} f)(y) - \int_{\mathbb{R}} L_m^X(t,y) \, f(y) \, \myd g_{n,k}(y) \\
		&= \int_{\mathbb{R}} L_m^X(t,y) \, \myd (g_{n,k} f)(y) - \frac{1}{2}\int_{\mathbb{R}} L_+^X(t,y) \, \myd g_{n,k}(y), 
\end{split}\]
$t < S^Y_{G(\mathbb{R})}, \ k \in \mathbb{N}, \ n \in \{0,\ldots,|F|\} \cap \mathbb{N}_0, \ \mathbf{P}\text{-a.s.}$, where in the last step we have applied Lemma \ref{lemma:connection_loc_times}. Finally, (\ref{eqn:zwischenrechnung78}) becomes
\[
	G_{n,k}(X_t) = G_{n,k}(X_0) + \int_0^t g_{n,k}(X_s)\, b(X_s) \, \myd B_s + \int_{\mathbb{R}} L_m^X(t,y) \, \myd (g_{n,k} f)(y), 
\]
$t < S^Y_{G(\mathbb{R})}, \ k \in \mathbb{N}, \ n \in \{0, \ldots, |F|\} \cap \mathbb{N}_0, \ \mathbf{P}\text{-a.s.}$ The integrator in the last term on the right-hand side has the following structure
\[
	g_{n,k}f = \ind_{\bigcup\limits_{i=0}^n \left[p_k^i,q_k^i \right)} = \sum_{i=0}^n \ind_{\left[p_k^i,q_k^i \right)}.
\]
Hence, calculating this integral, we obtain
\begin{equation}\label{eqn:approx_G_nk}
	G_{n,k}(X_t) = G_{n,k}(X_0) + \int_0^t g_{n,k}(X_s)\, b(X_s) \, \myd B_s + \sum_{i=0}^n \left(L_m^X(t,p_k^i) - L_m^X(t,q_k^i) \right),
\end{equation}
$t < S^Y_{G(\mathbb{R})}, \ k \in \mathbb{N}, \ n \in \{0, \ldots, |F|\} \cap \mathbb{N}_0, \ \mathbf{P}\text{-a.s.}$ \\
\indent Now, for arbitrary $n \in \{0, \ldots, |F|\} \cap \mathbb{N}_0$ we pass to the limit $k \rightarrow +\infty$. We observe
\[
	g_{n,k} \leq g_{n,k+1} \leq \frac{1}{f}, \qquad k \in \mathbb{N}, 
	\qquad \text{and} \qquad 
	\lim_{k \rightarrow +\infty} g_{n,k} = g_n \qquad \lambda\text{-a.e.}\footnote{$\lambda$ denotes the Lebesgue measure.}
\]
Together with Lemma \ref{lemma:integrability} this implies that the conditions of Theorem \ref{theorem:convergence} are fulfilled. For every $t \geq 0$ we obtain in probability
\[
	\lim_{k \rightarrow +\infty} \int_0^t g_{n,k}(X_s) \, b(X_s) \, \myd B_s = \int_0^t g_{n}(X_s) \, b(X_s) \, \myd B_s 
	\qquad \text{on } \{t < S^Y_{G(\mathbb{R})}\}.
\]
Since clearly
\[
	\lim_{k \rightarrow +\infty} G_{n,k}(x) = G_n(x), \qquad x \in \mathbb{R},
\]
and because of the continuity properties of $L_m^X$, for every $t \geq 0$, in addition we get in probability
\[\begin{split}
	& \lim_{k \rightarrow +\infty} \left(G_{n,k}(X_t) - G_{n,k}(X_0) - \sum_{i=0}^n \left(L_m^X(t,p_k^i) - L_m^X(t,q_k^i)\right)\right) \\
	& \phantom{======} = G_{n}(X_t) - G_{n}(X_0) - \sum_{i=0}^n \left(L_m^X(t,a_i) - L_m^X(t,b_i-)\right)
								    \qquad \text{on } \{t < S^Y_{G(\mathbb{R})}\}.
\end{split}\]
Finally, passing to the limit $k \rightarrow +\infty$ in (\ref{eqn:approx_G_nk}), for every $t \geq 0$ we conclude 
\[
	G_{n}(X_t) = G_{n}(X_0) + \int_0^t g_{n}(X_s) \, b(X_s) \, \myd B_s + \sum_{i=0}^n \left(L_m^X(t,a_i) - L_m^X(t,b_i-)\right)
\]
$\mathbf{P}$-a.s. on $\{t < S^Y_{G(\mathbb{R})}\}$. From the continuity in $t$ of the involved processes and by the arbitrariness of $n \in \{0, \ldots, |F|\} \cap \mathbb{N}_0$ it follows
\begin{equation}\label{eqn:limit_k}
	G_{n}(X_t) = G_{n}(X_0) + \int_0^t g_{n}(X_s) \, b(X_s) \, \myd B_s + \sum_{i=0}^n \left(L_m^X(t,a_i) - L_m^X(t,b_i-)\right),
\end{equation}
$t < S^Y_{G(\mathbb{R})}$, $n \in \{0, \ldots, |F|\} \cap \mathbb{N}_0$, $\mathbf{P}\text{-a.s.}$ \\
\indent In case of a finite set $F$, i.e., $|F| \in \mathbb{N}_0$, the claim of the theorem is proven. We only need to choose $n=|F|$. Then, since $g_{|F|} = \frac{1}{f}$ $\lambda$-a.e. and $G_{|F|} = G$, from (\ref{eqn:limit_k}) we conclude the desired form
\[
	Y_t = Y_0 + \int_0^t \sigma(Y_s) \, \myd B_s + \sum_{i = 0}^{|F|} \left(L_m^X(t,a_i) - L_m^X(t,b_i-) \right), 
	\qquad t < S^Y_{G(\mathbb{R})}, \ \mathbf{P}\text{-a.s.}
\]
\indent In case $|F| = +\infty$, now we pass to the limit $n \rightarrow +\infty$ in (\ref{eqn:limit_k}). By the definition of the functions $g_n$, $n \in \mathbb{N}_0$, we have
\[
	g_n \leq g_{n+1} \leq \frac{1}{f}, \quad  n \in \mathbb{N}_0, \qquad \text{and} \qquad	
	\lim_{n \rightarrow +\infty} g_n = \frac{1}{f} \qquad \lambda\text{-a.e.}
\]
Again, together with Lemma \ref{lemma:integrability} the assumptions of Theorem \ref{theorem:convergence} are fulfilled. For every $T \geq 0$ we derive 
\[
	\lim_{n \rightarrow +\infty} \sup_{0 \leq t \leq T} \left| \int_0^t g_n(X_s)\, b(X_s) \, \myd B_s -  \int_0^t \frac{b}{f}(X_s) \, \myd B_s\right| 
		= 0
\]
in probability on $\{T < S^Y_{G(\mathbb{R})}\}$. Since for all $T \geq 0$
\[
	\lim_{n \rightarrow +\infty} \sup_{x\in[-K,K]} |G_n(x) - G(x)| = 0, \qquad K \in\mathbb{N}\,,
\]
holds true, from (\ref{eqn:limit_k}) it follows
\[
	\lim_{n \rightarrow +\infty} \sup_{0 \leq t \leq T} \left| \sum_{i=0}^n \left(L_m^X(t,a_i) - L_m^X(t,b_i-)\right) - \left(G(X_t) - G(X_0) - \int_0^t \frac{b}{f}(X_s)\, \myd B_s\right) \right| = 0
\]
in probability on $\{T < S^Y_{G(\mathbb{R})}\}$. Defining the $\mathbb{F}$-adapted process 
\[
	\sum_{i = 0}^{+\infty} \left(L_m^X(t,a_i) - L_m^X(t,b_i-) \right) := G(X_t) - G(X_0) - \int_0^t \frac{b}{f}(X_s)\, \myd B_s, \qquad t \geq 0\,,
\]
we can conclude
\[
	Y_t = Y_0 + \int_0^t \sigma(Y_s) \, \myd B_s + \sum_{i = 0}^{+\infty} \left(L_m^X(t,a_i) - L_m^X(t,b_i-) \right), 
	\qquad t < S^Y_{G(\mathbb{R})}, \ \mathbf{P}\text{-a.s.},
\]
and the claim is also proven in case $|F| = +\infty$.
\end{proof}

Note that in case of Eq. (\ref{eqn:SDE_mvd}) expressed as an equation of type (\ref{eqn:SDE_mvasd}) the drift function is $f_\nu$ as introduced before (\ref{eqn:drift_measure}) and the corresponding set $F$ is empty. Hence, Theorem \ref{theorem:general_structure_space_transformation} contains the result (see e.g. \cite{engelbert_schmidt:1985}, Proposition 1) that Eq. (\ref{eqn:SDE_mvd}) can be transformed to an equation (\ref{eqn:SDE_without_drift}) without drift. However, the special feature of Eq. (\ref{eqn:SDE_mvasd}) is that, entailed by the set of singularities $F$, after applying the space transformation $G$ it cannot be excluded that there remains a drift term. \\
\indent As announced in the introduction, in our investigation we want to stay in the framework of continuous semimartingales. In the semimartingale case we want to concretise the remaining drift appearing in the decomposition of $Y = G(X)$. We point out in Theorem \ref{theorem:a_prioiri_semimart}, if $(Y,\mathbb{F}) = (G(X),\mathbb{F})$ is a continuous semimartingale up to $S^Y_{G(\mathbb{R})}$ then the decomposition of $Y$ has the form
\begin{equation}\label{eqn:SDE_transformed_eqn}
	Y_t = Y_0 + \int_0^t\sigma(Y_s) \, \myd B_s + \int_0^t \ind_{G(F_-)}(Y_s)\, \myd Y_s, \qquad t < S^Y_{G(\mathbb{R})}, \ \mathbf{P}\text{-a.s.}
\end{equation}
Consequently, under the a priori knowledge that $Y$ is a semimartingale up to $S^Y_{G(\mathbb{R})}$, in general, there remains a drift part that lives on the times when the process $Y$ is in the set $G(F_-)$. Note, if $(Y,\mathbb{F})$ is a continuous semimartingale up to $S^Y_{G(\mathbb{R})}$, then the term $\int_0^\cdot \ind_{G(F_-)}(Y_s)\, \myd Y_s$ in (\ref{eqn:SDE_transformed_eqn}) is actually a drift, i.e., of locally bounded variation. Indeed, let
\begin{equation}\label{eqn:semimartdecomp_of_Y}
	Y_t = Y_0 + M_t + V_t, \qquad t < S^Y_{G(\mathbb{R})}, \ \mathbf{P}\text{-a.s.}
\end{equation}
be the unique continuous semimartingale decomposition (see (\ref{eqn:semi_decomposition})) of $Y$. Since $F_-$ and therefore $G(F_-)$ are of Lebesgue measure zero, it follows immediately
\begin{equation}\label{eqn:drift_int_with_respect_to_process}
	\begin{split}
		\int_0^t \ind_{G(F_-)}(Y_s) \, \myd Y_s & = \int_0^t \ind_{G(F_-)}(Y_s) \, \myd M_s + \int_0^t \ind_{G(F_-)}(Y_s) \, \myd V_s \\
																						& = \int_0^t \ind_{G(F_-)}(Y_s) \, \myd V_s, \qquad t < S^Y_{G(\mathbb{R})}, \ \mathbf{P}\text{-a.s.}
	\end{split}
\end{equation}
Furthermore, taking (\ref{eqn:SDE_transformed_eqn}) into account, by the uniqueness of the continuous semimartingale decomposition of $Y$ we conclude
\[
	V_t = \int_0^t \ind_{G(F_-)}(Y_s) \, \myd Y_s, \qquad t < S^Y_{G(\mathbb{R})}, \ \mathbf{P}\text{-a.s.}
\]
\indent To handle the semimartingale case, we consider (\ref{eqn:SDE_transformed_eqn}) as a self-contained equation and fix the following definition of a solution.
\begin{defi}\label{def:sol_trans_eqn}
	A continuous $(\overline{\mathbb{R}}, \mathscr{B}(\overline{\mathbb{R}}))$-valued stochastic process $(Y,\mathbb{F})$ 
	defined on a probability space $(\Omega,\mathcal{F},\mathbf{P})$ is called a solution (up to the first exit from the open interval $G(\mathbb{R})$) 
	of Eq. (\ref{eqn:SDE_transformed_eqn})	if the following conditions are fulfilled: \medskip

	(i)   $Y_0 \in G(\mathbb{R})$. \medskip
	
	(ii)  $Y_t = Y_{t \wedge S^Y_{G(\mathbb{R})}}$, $t \geq 0$, $\mathbf{P}$-a.s. \medskip
	
	(iii) $(Y,\mathbb{F})$ is a semimartingale up to $S^Y_{G(\mathbb{R})}$. \medskip
	
	(iv)  There exists a Wiener process $(B,\mathbb{F})$ such that Eq. (\ref{eqn:SDE_transformed_eqn}) is satisfied for all 
								 					 $t < S^Y_{G(\mathbb{R})}$ $\mathbf{P}$-a.s. 
\end{defi}

Before we characterize solutions $(Y,\mathbb{F}) = (G(X),\mathbb{F})$ of Eq. (\ref{eqn:SDE_transformed_eqn}) by its semimartingale property, we show the following
\begin{lemma}\label{lemma:relation_loc_times}
	Let $(X,\mathbb{F})$ be a solution of Eq. (\ref{eqn:SDE_mvasd}). If the transformed process $(Y,\mathbb{F}) = (G(X), 
	\mathbb{F})$ is a continuous semimartingale up to $S^Y_{G(\mathbb{R})}$, then we have
	\[
		L_m^X(t,x\pm) = \frac{1}{2} \, L^Y_\pm(t,G(x)), \qquad t <  S^Y_{G(\mathbb{R})}, \ x \in \mathbb{R}, \ \mathbf{P}\text{-a.s.}
	\]
\end{lemma}
\begin{proof} Obviously, the requirements of Lemma \ref{lemma:conversion_loc_time} are satisfied and together with Lemma \ref{lemma:connection_loc_times} we can conclude 
\[
	L_\pm^Y(t, G(x)) = L_\pm^X(t,x) \, \frac{1}{f(x\pm)} = 2 \, L_m^X(t,x\pm), 
	\qquad t < S^Y_{G(\mathbb{R})}, \ x\in\mathbb{R}\setminus F, \ \mathbf{P}\text{-a.s.}
\]
Finally, the continuity properties of $L_\pm^Y$ and $L_m^X$ and the fact that $F$ is of Lebesgue measure zero imply the claim.
\end{proof}
\begin{theorem}\label{theorem:a_prioiri_semimart}
	Let $(X,\mathbb{F})$ be a solution of Eq. (\ref{eqn:SDE_mvasd}). Then $(Y,\mathbb{F}) = (G(X), \mathbb{F})$ is a solution of Eq. 
	(\ref{eqn:SDE_transformed_eqn}) if and only if	$(Y,\mathbb{F})$ is a continuous semimartingale up to $S^Y_{G(\mathbb{R})}$.
\end{theorem}
\begin{proof} Clearly, the necessity of the stated condition holds because of Definition \ref{def:sol_trans_eqn}. \\
\indent Now we assume that $(Y,\mathbb{F}) = (G(X), \mathbb{F})$ is a continuous semimartingale up to $S^Y_{G(\mathbb{R})}$. Obviously, conditions (i) and (ii) of Definition \ref{def:sol_trans_eqn} are satisfied. To establish condition (iv) we work, contrary to the proof of Theorem \ref{theorem:general_structure_space_transformation}, with an alternative approximation of $1/f$. We define right-continuous functions $\widetilde{g}_n$, $n \in \mathbb{N}$, of locally finite variation by
\[
	\widetilde{g}_n := \frac{1}{f} \wedge n, \qquad n \in \mathbb{N},
\]
and denote their primitives by
\[
	\widetilde{G}_n(x) := \int_0^x \widetilde{g}_n(y) \, \myd y, \qquad x \in \overline{\mathbb{R}}\,.
\]
Arguing in the same way as in the proof of Theorem \ref{theorem:general_structure_space_transformation} in (\ref{eqn:zwischenrechnung78}) and the following calculations, we obtain
\begin{equation}\label{eqn:decomposition_tilde_G_n}\begin{split}
	\widetilde{G}_n(X_t) 
		&= \widetilde{G}_n(X_0) + \int_0^t \widetilde{g}_n(X_s) \, b(X_s) \, \myd B_s \\
		&\phantom{==========}		+ \int_0^t \widetilde{g}_n(X_s-) \int_\mathbb{R} L_m^X(\myd s,y) \, \myd f(y)
				+ \frac{1}{2} \int_\mathbb{R} L_+^X(t,y) \, \myd \widetilde{g}_n(y) \\
		&= \widetilde{G}_n(X_0) + \int_0^t \widetilde{g}_n(X_s) \, b(X_s) \, \myd B_s + \int_{\mathbb{R}} L_m^X(t,y) \, \myd (\widetilde{g}_nf)(y),
\end{split}\end{equation}
$t < S^Y_{G(\mathbb{R})}, \ n\in \mathbb{N}, \ \mathbf{P}\text{-a.s.}$ 
The relations
\[
	\widetilde{g}_n \leq \widetilde{g}_{n+1} \leq \frac{1}{f}, \quad n \in \mathbb{N}, \qquad \text{and} \qquad
	\lim_{n \rightarrow +\infty} \widetilde{g}_n = \frac{1}{f}\,,
\]
as well as Lemma \ref{lemma:integrability} show that Theorem \ref{theorem:convergence} can be applied. Hence, for every $t \geq 0$ we conclude 
\[
	\lim_{n \rightarrow +\infty} \int_0^t \widetilde{g}_n(X_s) \, b(X_s) \, \myd B_s 
		=	\int_0^t \frac{b}{f}(X_s) \, \myd B_s  = \int_0^t \sigma(Y_s) \, \myd B_s
\]
in probability on $\{t < S^Y_{G(\mathbb{R})}\}$. Moreover, it holds true
\[
	\lim_{n \rightarrow +\infty} \widetilde{G}_n(x) = G(x), \qquad x \in \overline{\mathbb{R}}.
\]
Thus, for every $t \geq 0$ from (\ref{eqn:decomposition_tilde_G_n}) we obtain 
\begin{equation}\label{eqn:limit_1}
	\lim_{n \rightarrow +\infty} \int_\mathbb{R} L_m^X(t,y)\, \myd (\widetilde{g}_n f)(y) 
		= Y_t - Y_0 - \int_0^t \sigma(Y_s) \, \myd B_s =: \Sigma_t
\end{equation}
in probability on $\{t < S^Y_{G(\mathbb{R})}\}$. Additionally, our assumption that $(Y,\mathbb{F})$ is a continuous semimartingale up to $S^Y_{G(\mathbb{R})}$ allows us to apply Lemma \ref{lemma:relation_loc_times}. Therefore, for every $n \in \mathbb{N}$ we get
\[\begin{split}
	\int_\mathbb{R} L_m^X(t,y) \, \myd (\widetilde{g}_n f)(y) 
	 & = \frac{1}{2} \int_\mathbb{R} L_+^Y(t,G(y)) \, \myd (\widetilde{g}_n f)(y) \\
	 & = \frac{1}{2} \int_{G(\mathbb{R})} L_+^Y(t,y) \, \myd ((\widetilde{g}_n f)\circ H)(y), 
	 \qquad t < S^Y_{G(\mathbb{R})}, \ \mathbf{P}\text{-a.s.}
\end{split}\]
The functions $(\widetilde{g}_n f)\circ H$, $n \in \mathbb{N}$, appearing as integrators in the last expression are right-continuous and of locally bounded variation on $G(\mathbb{R})$. This implies that the continuous mappings $I_n : \overline{\mathbb{R}} \rightarrow \overline{\mathbb{R}}$, $n \in \mathbb{N}$, defined by
\[
	I_n(x) := \int_0^x ((\widetilde{g}_n f)\circ H)(y) \, \myd y, \qquad x \in \overline{\mathbb{R}}, \ n \in \mathbb{N},
\]
can be expressed as the difference of convex functions on $G(\mathbb{R})$. Therefore, applying the generalized It\^o formula (\ref{eqn:gen_ito_formula}) and using (\ref{eqn:semimartdecomp_of_Y}), we conclude
\begin{equation}\label{eqn:inter_step}\begin{split}
	& \int_{\mathbb{R}} L_m^X(t,y) \, \myd (\widetilde{g}_n f)(y) \\
		&\phantom{====}= \frac{1}{2} \int_{G(\mathbb{R})} L_+^Y(t,y) \, \myd ((\widetilde{g}_n f)\circ H)(y) \\
		&\phantom{====}= I_n(Y_t) - I_n(Y_0) - \int_0^t ((\widetilde{g}_n f) \circ H)(Y_s -) \, \myd M_s 
			- \int_0^t ((\widetilde{g}_n f) \circ H)(Y_s -) \, \myd V_s,
\end{split}\end{equation}
$t < S^Y_{G(\mathbb{R})}$, $\mathbf{P}$-a.s. Moreover, because of the relations
\[
	((\widetilde{g}_nf)\circ H)(y-) \, \ind_{G(\mathbb{R})}(y) 
			= \min\{1 , n \, (f \circ H)(y-) \} \, \ind_{G(\mathbb{R})}(y) 
			\leq \ind_{G(\mathbb{R})}(y),
\]
$y \in \mathbb{R}, \ n \in \mathbb{N},$ and
\[
	\lim_{n \rightarrow +\infty} ((\widetilde{g}_nf)\circ H)(y-) \, \ind_{G(\mathbb{R})}(y)
			= \ind_{\left\{x \in G(\mathbb{R}) :\, (f \circ H)(x-) > 0\right\}}(y), \qquad y \in \mathbb{R},
\]
as well as
\[
	\ind_{\{x \in G(\mathbb{R}) :\, (f \circ H)(x-) > 0\}} = \ind_{G(\mathbb{R})} \qquad \lambda\text{-a.e.},
\]
where we used the fact that $F_-$ and hence $G(F_-)$ have Lebesgue measure zero, we can apply Theorem \ref{theorem:convergence} again. For every $t \geq 0$ we conclude
\[
	\lim_{n \rightarrow +\infty} \int_0^t ((\widetilde{g}_n f) \circ H)(Y_s -) \, \myd M_s
			= \int_0^t  \ind_{G(\mathbb{R})} (Y_s) \, \myd M_s = M_t
\]
in probability on $\{t < S^Y_{G(\mathbb{R})}\}$. Since additionally it holds $\widetilde{g}_n f \leq \widetilde{g}_{n+1} f$, $n \in \mathbb{N}$, we have
\[
	\lim_{n \rightarrow +\infty} I_n(x) = \int_0^x \ind_{G(\mathbb{R})}(y)\, \myd y = x, \qquad x \in G(\mathbb{R}),
\]
and
\[\begin{split}
	\lim_{n \rightarrow +\infty} \int_0^t ((\widetilde{g}_n f) \circ H)(Y_s-) \, \myd V_s 
		& = \int_0^t \ind_{\{ x \in G(\mathbb{R}) : (f \circ H)(x-) > 0\}} (Y_s) \, \myd V_s \\
		& = \int_0^t \ind_{G(\mathbb{R}) \setminus G(F_-)} (Y_s) \, \myd V_s, 
		\qquad t < S^Y_{G(\mathbb{R})}, \ \mathbf{P}\text{-a.s.}
\end{split}\]
Summarizing these observations, from (\ref{eqn:inter_step}) we obtain
\[\begin{split}
	&\lim_{n \rightarrow +\infty} \int_{\mathbb{R}} L_m^X(t,y) \, \myd (\widetilde{g}_n f)(y) \\
		&\phantom{=,}= \lim_{n \rightarrow +\infty} \left( I_n(Y_t) - I_n(Y_0) - \int_0^t ((\widetilde{g}_n f) \circ H)(Y_s -) \, \myd M_s 
			- \int_0^t ((\widetilde{g}_n f) \circ H)(Y_s -) \, \myd V_s \right)\\
		&\phantom{=,}= Y_t - Y_0 - M_t - \int_0^t \ind_{G(\mathbb{R}) \setminus G(F_-)} (Y_s) \, \myd V_s \\
		&\phantom{=,}= \int_0^t \ind_{G(F_-)} (Y_s) \, \myd V_s
\end{split}\]
in probability on $\{t < S^Y_{G(\mathbb{R})}\}$. Comparing this result with (\ref{eqn:limit_1}) and using the continuity of the involved processes, we deduce
\[
	\Sigma_t = \int_0^t \ind_{G(F_-)} (Y_s) \, \myd V_s, \qquad t < S^Y_{G(\mathbb{R})}, \ \mathbf{P}\text{-a.s.}
\]
Therefore, $(\Sigma,\mathbb{F})$ is a continuous process of locally bounded variation on $[0,S^Y_{G(\mathbb{R})})$. From the uniqueness of the continuous semimartingale decomposition it follows
\[
	Y_t = Y_0 + \int_0^t \sigma(Y_s) \, \myd B_s + \int_0^t \ind_{G(F_-)} (Y_s) \, \myd V_s,
	\qquad t < S^Y_{G(\mathbb{R})}, \ \mathbf{P}\text{-a.s.},
\]
which we can rewrite (see (\ref{eqn:drift_int_with_respect_to_process})) as
\[
	Y_t = Y_0 + \int_0^t \sigma(Y_s) \, \myd B_s + \int_0^t \ind_{G(F_-)} (Y_s) \, \myd Y_s,
	\qquad t < S^Y_{G(\mathbb{R})}, \ \mathbf{P}\text{-a.s.},
\]
and the proof is finished. 
\end{proof}

We now pass to the inverse $H$ of the space transformation $G$ which can be easily handled as we demonstrate in the following
\begin{theorem}\label{theorem:spacetrans_H}
	Let $(Y,\mathbb{F})$ be a solution of Eq. (\ref{eqn:SDE_transformed_eqn}). Then $(X,\mathbb{F})$ given by $X = H(Y)$ is a solution to 
	Eq. (\ref{eqn:SDE_mvasd}).
\end{theorem}
\begin{proof} Clearly, condition (i) and, since
\[
	S^Y_{G(\mathbb{R})} = \inf\{t \geq 0: Y_t \notin G(\mathbb{R})\} 
								  = \inf\{t \geq 0: X_t \notin \mathbb{R}\} = S_\infty^X\,,
\]
condition (ii) of Definition \ref{def:solution} are satisfied by $X$. From the representation (\ref{eqn:representation_H}) of $H$ we see that $H$ restricted to $G(\mathbb{R})$ is the difference of convex functions. Therefore, applying the generalized It\^o formula (\ref{eqn:gen_ito_formula}), we obtain that $(X,\mathbb{F})$ is a continuous semimartingale up to $S_\infty^X$ with decomposition
\[\begin{split}
	X_t &= H(Y_0) + \int_0^t (f \circ H)(Y_s -) \, \myd Y_s + \frac{1}{2} \int_{G(\mathbb{R})} L_+^Y(t,y) \, \myd (f \circ H)(y)\\
			&= H(Y_0) + \int_0^t (f \circ H)(Y_s -) \, \sigma(Y_s) \, \myd B_s \\
			&\phantom{=====} + \int_0^t (f \circ H)(Y_s -) \, \ind_{G(F_-)}(Y_s) \, \myd Y_s
								+ \frac{1}{2} \int_{\mathbb{R}} L_+^Y(t,G(y)) \, \myd f(y), \qquad t < S_\infty^X, \ \mathbf{P}\text{-f.s.}
\end{split}\]
Using $(f \circ H)(y -) \, \ind_{G(F_-)}(y) = 0$, $y \in \mathbb{R}$, we see that the third summand on the right-hand side vanishes. For treating the second summand on the right-hand side, we remark that $f$ has at most countably many discontinuities. Recalling the definition of $\sigma$, we obtain
\[\begin{split}
	\int_0^t (f \circ H)(Y_s-)\, \sigma(Y_s) \, \myd B_s &= \int_0^t (f \circ H)(Y_s)\, \sigma(Y_s) \, \myd B_s \\
																											 &= \int_0^t (f \circ H)(Y_s)\, (b \circ H)(Y_s) \, (f \circ H)^{-1}(Y_s) \, \myd B_s \\
																											 &= \int_0^t (b \circ H)(Y_s) \, \myd B_s, \qquad t < S_\infty^X, \ \mathbf{P}\text{-f.s.},
\end{split}\]
where in the last step we used additionally that $G(F_+)$ is of Lebesgue measure zero. Hence, setting
\[
	L_m^X(t,y) := \frac{1}{2}\, L_+^Y(t,G(y)), \qquad (t,y) \in [0,S_\infty^X) \times \mathbb{R},
\]
we can write
\[
	X_t = X_0 + \int_0^t b(X_s) \, \myd B_s + \int_\mathbb{R} L_m^X(t,y) \, \myd f(y), \qquad t < S_\infty^X,\ \mathbf{P}\text{-a.s.}
\]
Since $(X,\mathbb{F})$ satisfies the conditions of Lemma \ref{lemma:conversion_loc_time}, we obtain
\[
	L_+^X(t,y) = L_+^Y(t,G(y))\,f(y) = L_m^X(t,y) \, 2f(y), \qquad t < S_\infty^X, \ y \in \mathbb{R}, \ \mathbf{P}\text{-a.s.}
\]
Using the occupation times formula (\ref{eqn:occupationtime}), for every non-negative measurable function $h$ it follows
\[\begin{split}
	\int_0^t h(X_s) \, \myd \assPro{X}_s &= \int_\mathbb{R} h(y) \, L_+^X(t,y) \, \myd y \\
																			 &= \int_\mathbb{R} h(y) \, L_m^X(t,y) \, m(\myd y), \qquad t < S_\infty^X, \ \mathbf{P}\text{-a.s.}
\end{split}\]
Clearly, $L_m^X$ also fulfils the desired continuity properties of Definition \ref{def:solution}(iv).
\end{proof}
\begin{remark}\label{remark:occupation_time_in_zero}
Replacing $\sigma$ by the real-valued coefficient $\widetilde{\sigma}$ defined in Remark \ref{remark:diffcoeff_infty} does not change Eq. (\ref{eqn:SDE_transformed_eqn}). This follows from the fact that any solution $(Y,\mathbb{F})$ of Eq. (\ref{eqn:SDE_transformed_eqn}) (for $\sigma$ as well as for $\widetilde{\sigma}$) has no occupation time in $G(F_+ \cap \{b \neq 0\})$, which can be verified by the occupation times formula (cf. proof of (\ref{eqn:no_occ_time})). For this it is important that, as $\sigma$, the coefficient $\widetilde{\sigma}$ is different from zero on $G(F_+ \cap \{b \neq 0\})$. Otherwise, there would be also allowed solutions which are \emph{sticky} in $G(F_+ \cap \{b \neq 0\})$, i.e., solutions with a strictly positive occupation time in this set. However, in this paper we will not deal with sticky solutions. \mydia
\end{remark}

Motivated by our intention to keep within the class of semimartingales, we now introduce the notion of a \emph{good} solution of Eq. (\ref{eqn:SDE_mvasd}).
\begin{defi}\label{defi:good_solution}
	We say that a solution $(X,\mathbb{F})$ of Eq. (\ref{eqn:SDE_mvasd}) is \emph{good} if $(Y,\mathbb{F}) = (G(X),\mathbb{F})$ is a 
	continuous semimartingale up to $S^Y_{G(\mathbb{R})}$.
\end{defi}
\noindent Now the question arises: Under which conditions is $(X,\mathbb{F})$ a good solution? As a first result we stress the role of the process $\left( \bigl(\sum_{i = 0}^{|F|} (L_m^X(t,a_i) - L_m^X(t,b_i-) )\bigr)_{t \geq 0}, \mathbb{F} \right)$ which appears in Theorem \ref{theorem:general_structure_space_transformation} in the decomposition of $(Y,\mathbb{F})$. The following characterization of a good solution is just a reformulation of Theorem \ref{theorem:a_prioiri_semimart}. Indeed, we only need to follow the proof of Theorem \ref{theorem:a_prioiri_semimart}.
\begin{theorem}\label{theorem:a_prioiri_semimart_alt_formulation}
	A solution $(X,\mathbb{F})$ of Eq. (\ref{eqn:SDE_mvasd}) is good if and only if $\sum_{i = 0}^{|F|} (L_m^X(\, . \,,a_i) - L_m^X(\, . \,,b_i-) )$ is 
	of locally bounded variation on $[0, S^Y_{G(\mathbb{R})})$. In that case, for $(Y,\mathbb{F}) = (G(X),\mathbb{F})$ it holds
	\[
		\int_0^t \ind_{G(F_-)}(Y_s) \, \myd Y_s 
							= \sum_{i = 0}^{|F|} \left( L_m^X(\, . \,,a_i) - L_m^X(\, . \,,b_i-) \right), \qquad t < S_{G(\mathbb{R})}^Y, \ \mathbf{P}\text{-a.s.}
	\]
\end{theorem}
Under additional assumptions on the set $F$ of singularities we can improve the preceding Theorem. For $F$ consisting only of isolated points we have
\begin{theorem}\label{theorem:isolated_points}
	Suppose that $|F \cap [-N,N]| < +\infty$, $N \in \mathbb{N}$. Then every solution $(X,\mathbb{F})$ of Eq. (\ref{eqn:SDE_mvasd}) is good. 
\end{theorem}
\begin{proof} To prove the result, we use Theorem \ref{theorem:a_prioiri_semimart_alt_formulation}. If $F$ is finite, then the finite sum
$\sum_{i = 0}^{|F|} (L_m^X(\, . \,,a_i) - L_m^X(\, . \,,b_i-) )$ is obviously of locally bounded variation on  $[0, S^Y_{G(\mathbb{R})})$. Therefore, $(G(X),\mathbb{F})$ is a solution of Eq. (\ref{eqn:SDE_transformed_eqn}). In case of an infinite set $F$ satisfying $|F \cap [-N,N]| < +\infty$, $N \in \mathbb{N}$, for arbitrary $N \in \mathbb{N}$ we define the $\mathbb{F}$-stopping time $T_N := \inf \{t \geq 0: X_t \notin [-N,N]\}$.
Using Corollary \ref{corr:comp_interval}, for arbitrary $t \geq 0$ it is
\[
	\lim_{n \rightarrow +\infty} \sum_{i=0}^n \left( L_m^X(t,a_i) - L_m^X(t,b_i-) \right)
			= \sum_{a_i, b_i \in [-N,N]} \left( L_m^X(t,a_i) - L_m^X(t,b_i-) \right), \qquad  t \leq T_N, \ \mathbf{P}\text{-a.s.}
\]
and the sum on the right-hand side is finite. Additionally, since $\{t \leq T_N\} \subseteq \{t < S_\infty^X\}$, by Theorem \ref{theorem:general_structure_space_transformation} we have
\[
	\lim_{n \rightarrow +\infty} \sum_{i=0}^n \left(L_m^X(t,a_i) - L_m^X(t,b_i-)\right)
			= \sum_{i=0}^{+\infty} \left(L_m^X(t,a_i) - L_m^X(t,b_i-)\right)
\]
in probability on $\{t \leq T_N\}$. Therefore, using the continuity of the involved processes, we conclude
\[
	\sum_{i=0}^{+\infty} \left(L_m^X(t,a_i) - L_m^X(t,b_i-)\right) = \sum_{a_i, b_i \in [-N,N]} \left( L_m^X(t,a_i) - L_m^X(t,b_i-) \right), \qquad 
	t \leq T_N, \	\mathbf{P}\text{-a.s.}
\]
But since $N\in \mathbb{N}$ was chosen arbitrarily and since $T_N \uparrow S_{G(\mathbb{R})}^Y$, $N \rightarrow +\infty$, holds, this means that $\left(\sum_{i=0}^{+\infty} \left(L_m^X(t,a_i) - L_m^X(t,b_i-)\right)\right)_{t \geq 0}$ is of locally bounded variation on $[0,S_{G(\mathbb{R})}^Y)$, and hence $(G(X),\mathbb{F})$ is a solution of Eq. (\ref{eqn:SDE_transformed_eqn}).
\end{proof}
Under the assumption that the set $F$ is \emph{countable}, we can get more insight in the structure of the drift part of $(Y,\mathbb{F}) = (G(X),\mathbb{F})$ for good solutions $(X,\mathbb{F})$ of Eq. (\ref{eqn:SDE_mvasd}).
\begin{proposition}\label{proposition:drift_part_pointwise_sum}
	Suppose that $F$ is countable. Let $(X,\mathbb{F})$ be a good solution of Eq. (\ref{eqn:SDE_mvasd}). 	Then, for the solution 
	$(Y,\mathbb{F})$ of Eq. (\ref{eqn:SDE_transformed_eqn}) defined by $Y = G(X)$ it holds for all $t < S^Y_{G(\mathbb{R})}$ 
	$\mathbf{P}$-a.s.
	\[
		\int_0^t \ind_{G(F_-)} (Y_s) \, \myd Y_s = \sum_{a \in F_-} \left(L_m^X(t,a) - L_m^X(t,a-) \right) 
																						 = \frac{1}{2} \sum_{a \in G(F_-)} \left(L_+^Y(t,a) - L_-^Y(t,a) \right).
	\]
\end{proposition}
\begin{proof}
\indent By $V$ we denote the process of locally bounded variation in the semimartingale decomposition of $Y$. Using (\ref{eqn:drift_int_with_respect_to_process}) and property (\ref{eqn:loc_time_and_variation_process}), we obtain
\begin{equation}\label{eqn:drift_part_pointwise_sum}
	\begin{split}
		\int_0^t \ind_{G(F_-)}(Y_s) \, \myd Y_s &= \int_0^t \ind_{G(F_-)}(Y_s) \, \myd V_s \\
																						&= \sum_{a \in G(F_-)} \int_0^t \ind_{\{a\}}(Y_s)\,\myd V_s \\
																						&= \frac{1}{2} \sum_{a \in G(F_-)} \left( L_+^Y(t,a) - L_-^Y(t,a)\right), 
																							 \qquad t < S_{G(\mathbb{R})}^Y, \ \mathbf{P}\text{-a.s.},
	\end{split}
\end{equation}
which together with Lemma \ref{lemma:relation_loc_times} ends the proof. 
\end{proof}
\begin{remark}
	The assumption of Proposition \ref{proposition:drift_part_pointwise_sum} that $F$ or, equivalently, 
  $F_-$ is countable is essential in (\ref{eqn:drift_part_pointwise_sum}). Indeed, if 
	the set of Lebesgue measure zero $F_-$ is uncountable, the same holds for $G(F_-)$. Hence, besides a singular discrete part, $\myd V_s$ can 
	have a singular continuous part. \mydia
\end{remark}
We now come to a characterization of good solutions $(X,\mathbb{F})$ of Eq. (\ref{eqn:SDE_mvasd}) for countable sets $F$ for which its accumulation points
\[
	F^A := \{ x \in F: \forall \, \varepsilon > 0 \; \exists \, y \in F \cap (x-\varepsilon, x+\varepsilon), \ y \neq x \}
\]
are isolated.
\begin{theorem}\label{theorem:finitely_many_accumulation_points}
	Suppose that $|F^A \cap [-N,N]| < +\infty$, $N \in \mathbb{N}$. Then $(X,\mathbb{F})$ is a good solution of Eq. (\ref{eqn:SDE_mvasd})
	if and only if the following conditions are fulfilled: If $F_-$ is infinite then \medskip
	
	(i) $\displaystyle\sum_{a \in F_-} \left| L_m^X(t,a) - L_m^X(t,a-) \right| < +\infty, \qquad t < S_\infty^X, \ \mathbf{P}\text{-a.s.}$ \medskip
	
	(ii) For any enumeration $\{a_1, a_2, \ldots\}$ of $F_-$, the sequence
	$\left(\sum_{i=1}^n (L_m^X(\,.\,, a_i) - L_m^X(\,.\,, a_i-))\right)_{n\in \mathbb{N}}$ of processes converges $\mathbf{P}$-a.s. locally in 
	variation on $[0, S_\infty^X)$.
\end{theorem}
\begin{proof}
The proof of this theorem is rather technical and is therefore omitted. The interested reader is referred to S. Blei \cite{blei_thesis_2010}, Satz 2.3.62.
\end{proof}
\indent To finish this section, we remark that by our observations, particularly Theorem \ref{theorem:a_prioiri_semimart_alt_formulation}, it is reasonable to conjecture that a complete analysis of Eq. (\ref{eqn:SDE_mvasd}) goes beyond the class of semimartingales and, consequently, should be envisaged in the richer class of local Dirichlet processes.

\section{Symmetric Solutions \--- Existence and Uniqueness}\label{sec:symmetric_solutions}
\noindent In the next two sections we proceed with the systematic investigation of existence and uniqueness of good solutions of Eq. (\ref{eqn:SDE_mvasd}). We need the following preparatory lemma, which compares the sets $N_b$ and $N_\sigma$ as well as $E_{b/\sqrt{f}}$ and $E_\sigma$, as introduced before Lemma \ref{lemma:stopping}.
\begin{lemma}\label{lemma:connection_sets}
	We have $N_\sigma^c = G(N_b^c)$ and $E_\sigma^c = G(E_{b/\sqrt{f}}^c)$.
\end{lemma}
\begin{proof} The first equality is obvious. To show the second assertion, we observe	
\[
	E_\sigma^c \subseteq (G(-\infty), G(+\infty))
\]
and, using (\ref{eqn:representation_H}),
\[\begin{split}
		\int_U \sigma^{-2}(y) \, \myd y\ & = \int_U (b \circ H)^{-2}(y) \,(f\circ H)(y) \, \myd H(y) \\
																		 &= \int_{H(U)} b^{-2}(y)\,f(y)\, \myd y 
\end{split}\]
for every open subset $U \subseteq (G(-\infty), G(+\infty))$. Since $G$ and $H$ are continuous on $\mathbb{R}$ and \linebreak[4] $(G(-\infty),G(+\infty))$, respectively, the proof is completed. 
\end{proof}

We recall that Corollary \ref{corollary:left_continuity} reveals that the local time $L_m^X$ of a solution $(X,\mathbb{F})$ of Eq. (\ref{eqn:SDE_mvasd}) is continuous in the state variable except in the points of $F_-$. In this section, as a first step, we are interested in solutions of Eq. (\ref{eqn:SDE_mvasd}) which possess a continuous local time $L^X_m$. 
\begin{defi}
	A solution $(X,\mathbb{F})$ of Eq. (\ref{eqn:SDE_mvasd}) is called \emph{symmetric} if its local time $L^X_m$ is continuous in the state variable.
\end{defi}
\noindent This notion is motivated by the obvious fact that the local time $L_m^X$ of a solution $(X,\mathbb{F})$ of Eq. (\ref{eqn:SDE_mvasd}) is continuous in the state variable if and only if the local time $L_m^X$ is symmetric, i.e., $L_m^X$ coincides with the symmetric local time $\hat{L}_m^X(t,x) := (L_m^X(t,x) + L_m^X(t,x-))/2$, $t < S_\infty^X$, $x \in \mathbb{R}$.
\begin{proposition}\label{prop:symmetric_good_solution_and_eqn_without_drift}
	(i) Let $F$ be countable. If $(X,\mathbb{F})$ is a symmetric good solution of Eq. (\ref{eqn:SDE_mvasd}), then $(Y,\mathbb{F}) = (G(X),\mathbb{F})$ 
	is a solution of Eq. (\ref{eqn:SDE_without_drift}) with $\sigma =(b/f) \circ H$. \smallskip
	
	(ii) Conversely, for arbitrary $F$ it holds: If $(Y,\mathbb{F})$ is a solution of Eq. (\ref{eqn:SDE_without_drift}) with 
	$\sigma =(b/f) \circ H$, then $(X,\mathbb{F}) = (H(Y),\mathbb{F})$ is a symmetric good solution of Eq. (\ref{eqn:SDE_mvasd}). \smallskip
	
	(iii) Suppose $|F^A \cap [-N,N]| < +\infty$, $N \in \mathbb{N}$. Then any symmetric solution $(X,\mathbb{F})$ of Eq. (\ref{eqn:SDE_mvasd}) is also a 
	good solution, and hence $(Y,\mathbb{F}) = (G(X),\mathbb{F})$ is a solution of Eq. (\ref{eqn:SDE_without_drift}) with $\sigma =(b/f) 
	\circ H$.
\end{proposition}
Before we prove Proposition \ref{prop:symmetric_good_solution_and_eqn_without_drift}, we give the auxiliary
\begin{lemma}\label{lemma:eqn_without_drift_and_transformed_symmetric_sol}
	The process $(Y,\mathbb{F})$ is a solution of Eq. (\ref{eqn:SDE_without_drift}) with $\sigma =(b/f) \circ H$ if and only if $(Y,\mathbb{F})$ is a 
	solution of Eq. (\ref{eqn:SDE_transformed_eqn}) satisfying 
	\begin{equation}\label{eqn:aux_zero_drift}
		\int_0^t \ind_{G(F_-)} (Y_s) \, \myd Y_s = 0, \qquad t < S^Y_{G(\mathbb{R})}, \ \mathbf{P}\text{-a.s.}
	\end{equation}
\end{lemma}
\begin{proof}
Let $(Y,\mathbb{F})$ be a solution of Eq. (\ref{eqn:SDE_without_drift}) with $\sigma =(b/f) \circ H$. Since $E_\sigma^c \subseteq G(\mathbb{R})$, from Lemma \ref{lemma:stopping} it follows $Y_t =Y_{t \wedge S^Y_{G(\mathbb{R})}}$, $t \geq 0$, $\mathbf{P}$-a.s. Moreover, $G(F_-)$ is of Lebesgue measure zero, and hence (\ref{eqn:aux_zero_drift}) follows immediately. Therefore, $(Y,\mathbb{F})$ is also a solution of Eq. (\ref{eqn:SDE_transformed_eqn}).\\
\indent Conversely, if $(Y,\mathbb{F})$ is a solution of Eq. (\ref{eqn:SDE_transformed_eqn}) satisfying (\ref{eqn:aux_zero_drift}), then we have
\[
	Y_t = Y_0 + \int_0^t \sigma(Y_s) \, \myd B_s, \qquad t < S^Y_{G(\mathbb{R})}, \ \mathbf{P}\text{-a.s.},
\]
and by similar arguments as used in the proof of \cite{engelbert_schmidt:1989_III}, Proposition (4.29), it can be shown that for arbitrary $t \geq 0$ the last equality also holds on $\{S^Y_{G(\mathbb{R})} \leq t < S^Y_\infty\}$. Hence, $(Y,\mathbb{F})$ is a solution of Eq. (\ref{eqn:SDE_without_drift}).
\end{proof}
\begin{proof}[Proof of Proposition \ref{prop:symmetric_good_solution_and_eqn_without_drift}] Assertion (i) follows directly from Proposition \ref{proposition:drift_part_pointwise_sum} and Lemma \ref{lemma:eqn_without_drift_and_transformed_symmetric_sol}. \\
\indent If $(Y,\mathbb{F}) = (G(X),\mathbb{F})$ is a solution of Eq. (\ref{eqn:SDE_without_drift}) with $\sigma =(b/f) \circ H$, then Lemma \ref{lemma:eqn_without_drift_and_transformed_symmetric_sol} implies that $(Y,\mathbb{F})$ is also a solution of Eq. (\ref{eqn:SDE_transformed_eqn}). Hence, from Theorem \ref{theorem:spacetrans_H} we obtain that $(X,\mathbb{F}) = (H(Y),\mathbb{F})$ is a good solution of Eq. (\ref{eqn:SDE_mvasd}). Moreover, combining (\ref{eqn:loc_time_and_variation_process}) for $Y$ and Lemma \ref{lemma:relation_loc_times}, we see that $(X,\mathbb{F})$ is also symmetric and (ii) is proven. \\
\indent Statement (iii) can now be deduced with the help of Theorem \ref{theorem:finitely_many_accumulation_points}. Indeed, if $(X,\mathbb{F})$ is a symmetric solution of Eq. (\ref{eqn:SDE_mvasd}) then the sums appearing in the conditions (i) and (ii) of Theorem \ref{theorem:finitely_many_accumulation_points} are $\mathbf{P}$-a.s. equal to zero. Therefore, $(X,\mathbb{F})$ is good and we can apply (i).
\end{proof}
Concerning the existence of symmetric good solutions of Eq. (\ref{eqn:SDE_mvasd}), we can state the following
\begin{theorem}\label{theorem:ex_symmetric_solutions}
(i) Let $F$ be arbitrary. Suppose $E_{b/\sqrt{f}} \subseteq N_b$ is satisfied. Then for every initial distribution there exists a symmetric good solution of Eq. (\ref{eqn:SDE_mvasd}). \smallskip
	
(ii) Let $F$ be countable. Then for every initial distribution there exists a symmetric good solution of Eq. (\ref{eqn:SDE_mvasd}) if and only if the condition $E_{b/\sqrt{f}} \subseteq N_b$ is satisfied.
\end{theorem}
\begin{proof} To prove (i) and the sufficiency of the condition $E_{b/\sqrt{f}} \subseteq N_b$ in (ii), we choose an arbitrary initial distribution $\mu$. Lemma \ref{lemma:connection_sets} implies that $E_\sigma \subseteq N_\sigma$ holds, too. Hence, by Theorem \ref{theorem:e_u_gen_drift}(i), Eq. (\ref{eqn:SDE_without_drift}) with diffusion coefficient $\sigma = (b/f) \circ H$ possesses a solution $(Y,\mathbb{F})$ with initial distribution $\mu \circ G^{-1}$. Now from Proposition \ref{prop:symmetric_good_solution_and_eqn_without_drift} it follows immediately that $(X,\mathbb{F}) = (H(Y),\mathbb{F})$ is a symmetric good solution of Eq. (\ref{eqn:SDE_mvasd}) with initial distribution $\mu$. \\
\indent To prove the necessity of $E_{b/\sqrt{f}}\subseteq N_b$ in (ii), we fix $x_0 \in E_{b/\sqrt{f}}$ and take a symmetric good solution $(X,\mathbb{F})$ of Eq. (\ref{eqn:SDE_mvasd}) started at $X_0 = x_0$. From Proposition \ref{prop:symmetric_good_solution_and_eqn_without_drift}(i) we obtain that $(Y,\mathbb{F}) = (G(X),\mathbb{F})$ is a solution of Eq. (\ref{eqn:SDE_without_drift}) with diffusion coefficient $\sigma =(b/f) \circ H$ and initial value $Y_0 = y_0 = G(x_0)$. Moreover, via Lemma \ref{lemma:connection_sets} it follows $y_0 \in E_\sigma$. Hence, Lemma \ref{lemma:stopping} implies $Y_t = y_0$, $t \geq 0$, $\mathbf{P}$-a.s. Therefore, we conclude
\begin{equation}\label{eqn:belonging_to_zeros_of_sigma}
	0 = \int_0^t \sigma^2(Y_s) \,\myd s = \sigma^2(y_0) \, t, \qquad t \geq 0,
\end{equation}
from which we see $y_0 \in N_\sigma$. By the arbitrariness of $x_0 \in E_{b/\sqrt{f}}$ we conclude $G(E_{b/\sqrt{f}}) \subseteq N_\sigma$ and via Lemma \ref{lemma:connection_sets} we obtain $E_{b/\sqrt{f}} \subseteq N_b$.
\end{proof}

Now we treat the question of uniqueness of symmetric good solutions of Eq. (\ref{eqn:SDE_mvasd}).
\begin{theorem}\label{theorem:ex_u_un_symmetric_sol}
(i) Let $F$ be arbitrary. Suppose $E_{b/\sqrt{f}} \subseteq N_b$. If for every initial distribution Eq. (\ref{eqn:SDE_mvasd}) possesses a unique symmetric good solution, then it holds $E_{b/\sqrt{f}} = N_b$. \smallskip

(ii) Let $F$ be countable. Then for every initial distribution Eq. (\ref{eqn:SDE_mvasd}) possesses a unique symmetric good solution if and only if the condition $E_{b/\sqrt{f}} = N_b$ is satisfied.
\end{theorem}
\begin{proof}
First, we show (i) and the necessity of $E_{b/\sqrt{f}} = N_b$ in (ii). To this end, let us assume that, for every initial distribution, Eq. (\ref{eqn:SDE_mvasd}) possesses a unique symmetric good solution. In case of a countable set $F$ by means of Theorem \ref{theorem:ex_symmetric_solutions}(ii) this implies $E_{b/\sqrt{f}} \subseteq N_b$. For uncountable $F$ we suppose that $E_{b/\sqrt{f}} \subseteq N_b$ holds. Now we accomplish this part of the proof by contraposition. Let us assume that $E_{b/\sqrt{f}} = N_b$ does not hold, i.e., there exists an $x_0 \in E^c_{b/\sqrt{f}} \cap N_b$. Then, clearly, $\overline{X} \equiv x_0$ is a symmetric good solution of Eq. (\ref{eqn:SDE_mvasd}). On the other hand, we can also find a non-trivial symmetric good solution of Eq. (\ref{eqn:SDE_mvasd}) started at $x_0$. Indeed, because of Lemma \ref{lemma:connection_sets}, $E_{b/\sqrt{f}} \subseteq N_b$ implies $E_\sigma \subseteq N_\sigma$. Hence, there exists the so-called fundamental solution $(Y,\mathbb{F})$ of Eq. (\ref{eqn:SDE_without_drift}) with diffusion coefficient $\sigma = (b/f) \circ H$ started at $y_0=G(x_0) \in E^c_\sigma \cap N_\sigma$ and $Y$ is non-trivial (see \cite{engelbert_schmidt:1989_III}, Definition (4.16) and Theorem (4.17)). Finally, via Proposition \ref{prop:symmetric_good_solution_and_eqn_without_drift}(ii) we conclude that $(X,\mathbb{F}) = (H(Y),\mathbb{F})$ is a non-trivial symmetric good solution of Eq. (\ref{eqn:SDE_mvasd}) with initial point $x_0$. \\
\indent To show the sufficiency of $E_{b/\sqrt{f}} = N_b$ in (ii), we take two symmetric good solutions $(X^i,\mathbb{F}^i)$, $i =1,2$, of Eq. (\ref{eqn:SDE_mvasd}) defined on $(\Omega^i, \mathcal{F}^i, \mathbf{P}^i)$, $i =1,2$, which possess the same initial distribution. Then from Proposition \ref{prop:symmetric_good_solution_and_eqn_without_drift}(i) it follows that $(Y^i,\mathbb{F}^i) := (G(X^i), \mathbb{F}^i)$, $i = 1,2$, are two solutions of Eq. (\ref{eqn:SDE_without_drift}) with $\sigma = (b/f) \circ H$ and identical initial distributions. Now we take a sequence of intervals $[a_n,b_n]$, $n \in \mathbb{N}$, such that 
\begin{equation}\label{eqn:sequence_of_intervals}
	[a_n,b_n] \subseteq [a_{n+1},b_{n+1}] \subseteq G(\mathbb{R}),\ n \in \mathbb{N}, 
	\quad \text{ and } \quad 
	\bigcup_{n \in \mathbb{N}} [a_n,b_n] = G(\mathbb{R})
\end{equation}
and define the $\mathbb{F}$-stopping times $S^i_n := \inf\{t \geq 0: Y^i_t \notin (a_n, b_n)\}$, $n \in \mathbb{N}$, $i =1,2$. For arbitrary $n\in \mathbb{N}$ we set $Y^{i,n}_t := Y^i_{t \wedge S^i_n}$, $t \geq 0$. Then it holds
\begin{equation}\label{eqn:end_of_proof_uniqueness}
		Y^{i,n}_t = Y^i_0 + \int_0^t \sigma_n(Y^{i,n}_s) \, \myd B^i_s, \qquad t \geq 0, \ \mathbf{P}^i\text{-a.s.}, \ i=1,2,
\end{equation}
where $\sigma_n := \sigma \ind_{(a_n,b_n)}$ and $B^i$, $i = 1,2$, is the corresponding Wiener process. That means, $(Y^{i,n},\mathbb{F})$, $i =1,2$, are solutions to the same equation of type (\ref{eqn:SDE_without_drift}) with identical initial distribution. Moreover, via Lemma \ref{lemma:connection_sets} we obtain $E_\sigma = N_\sigma$, and hence we have $E_{\sigma_n} = N_{\sigma_n}$. Applying Theorem \ref{theorem:e_u_gen_drift}(ii), we conclude that $Y^{1,n}$ and $Y^{2,n}$ have the same distribution on $C_{\overline{\mathbb{R}}}([0,+\infty))$. This implies that the distributions of $Y^1$ and $Y^2$ coincide on $\mathcal{C}_{S_n-}$, where $\mathbb{C} = (\mathcal{C}_t)_{t \geq 0}$ is the filtration generated by the coordinate mappings $Z = (Z_t)_{t \geq 0}$,
\[
	Z_t(\omega) = \omega(t), \qquad \omega \in C_{\overline{\mathbb{R}}}([0,+\infty)), \ t \geq 0,
\]
and $S_n := \inf\{t \geq 0: Z_t \notin (a_n, b_n)\}$. Finally, it follows easily that the distributions of $Y^1$ and $Y^2$ coincide on $C_{\overline{\mathbb{R}}}([0,+\infty))$. Hence, we conclude that the distributions of $X^1$ and $X^2$ coincide, too. 
\end{proof}
If the set $F$ satisfies the condition that $F^A$ consists only of isolated points, then from Theorem \ref{theorem:ex_symmetric_solutions}(ii) and \ref{theorem:ex_u_un_symmetric_sol}(ii) and Proposition \ref{prop:symmetric_good_solution_and_eqn_without_drift}(ii) we get immediately
\begin{corollary}\label{cor:ex_u_un_isolated_points}
	 Suppose $|F^A \cap [-N,N]| < +\infty$, $N \in \mathbb{N}$. Then the following statements hold. \smallskip
	 
	(i) For every initial distribution there exists a symmetric solution of Eq. (\ref{eqn:SDE_mvasd}) if and only if the condition 
			$E_{b/\sqrt{f}} \subseteq N_b$ is satisfied. \smallskip
			 
	(ii) For every initial distribution there exists a unique symmetric solution of Eq. (\ref{eqn:SDE_mvasd}) if and only if the condition 
			 $E_{b/\sqrt{f}} = N_b$ is satisfied.
\end{corollary}

\section{Skew Solutions \--- Existence and Uniqueness}\label{sec:skew_solutions}
\noindent In this section, we want to study good solutions $(X,\mathbb{F})$ of Eq. (\ref{eqn:SDE_mvasd}) which, in general, do not possess a continuous local time $L_m^X$. We start with an important example of an equation of type (\ref{eqn:SDE_mvasd}) and illustrate a new feature of equations of type (\ref{eqn:SDE_mvasd}) compared with equations of type (\ref{eqn:SDE_mvd}). The example shows that, even in the case $E_{b/\sqrt{f}} = N_b$, or stronger $E_b = N_b$, there are, in general, solutions of Eq. (\ref{eqn:SDE_mvasd}) which are different in law. Note that the uniqueness result of Theorem \ref{theorem:ex_u_un_symmetric_sol} was achieved by considering only symmetric solutions of Eq. (\ref{eqn:SDE_mvasd}). The example will provide us with a whole variety of (good) solutions and will give us an idea for an approach to Eq. (\ref{eqn:SDE_mvasd}) more general than in Section \ref{sec:symmetric_solutions}. \\
\indent Fixing $\delta \in (1,2)$ and $x_0 \in \mathbb{R}$, a typical example of an equation of type (\ref{eqn:SDE_mvasd}) with generalized and singular drift is the Bessel equation which is satisfied for $x_0 \geq 0$ by the $\delta$-dimensional Bessel process. Classically, this is an equation with ordinary drift (see D. Revuz and M. Yor \cite{revuzyor}, Ch. XI, \S 1):
\begin{equation}\label{eqn:SDE_BES}
	X_t = x_0 + B_t + \int_0^t \frac{\delta - 1}{2\, X_s} \, \myd s\,,
\end{equation}
where the notion of a solution of Eq. (\ref{eqn:SDE_BES}) is introduced analogously to Definition (\ref{def:solution}), but condition (iv) is omitted. Using the drift function $f_\delta(x) = |x|^{\delta -1}$, $x \in \mathbb{R}$, Eq. (\ref{eqn:SDE_BES}) coincides with
\begin{equation}\label{eqn:SDE_BES_mvasd}
	X_t = x_0 + B_t + \int_\mathbb{R} L_m^X(t,y)\,\myd f_\delta(y)\,.
\end{equation}
Indeed, for any solution $(X,\mathbb{F})$ of Eq. (\ref{eqn:SDE_BES_mvasd}), taking $\myd f_\delta(y) = (\delta-1) \, y^{-1} \, f_\delta(y)\, \myd y$ and the occupation times formula of Definition \ref{def:solution}(iv) into account, with $m$ given by $m(\myd y) = 2 \, f_\delta(y) \, \myd y$, we obtain
\begin{equation}\label{eqn:reformulation_drift_p}
	\begin{split}
		\int_{(0,+\infty)} L_m^X(t,y) \, \myd f_\delta(y) 
							&= \int_{(0,+\infty)}  \frac{\delta - 1}{2 \, y} \, L_m^X(t,y) \, m(\myd y) \\
							&= \int_0^t \frac{\delta - 1}{2\, X_s} \, \ind_{(0,+\infty)} (X_s) \, \myd \assPro{X}_s \\
							& = \int_0^t \frac{\delta - 1}{2\, X_s} \, \ind_{(0,+\infty)} (X_s) \, \myd s\,, \qquad t < S_\infty^X, \ \mathbf{P}\text{-a.s.}
	\end{split}
\end{equation}
Analogously, we have
\begin{equation}\label{eqn:reformulation_drift_n}
	\int_{(-\infty,0)} L_m^X(t,y) \, \myd f_\delta(y) = \int_0^t \frac{\delta - 1}{2\, X_s} \, \ind_{(-\infty,0)} (X_s) \, \myd s\,,
	\qquad t < S_\infty^X, \ \mathbf{P}\text{-a.s.}
\end{equation}
On the other hand, it was shown in S. Blei \cite{blei_bessel_2011}, Proposition 2.19, that for any solution $(X,\mathbb{F})$ of Eq. (\ref{eqn:SDE_BES}) there exists a local time $L_m^X$ that satisfies the requirements of Definition \ref{def:solution}(iv) with respect to the drift function $f_\delta$. Therefore, for any solution $(X,\mathbb{F})$ of Eq. (\ref{eqn:SDE_BES}) the equalities (\ref{eqn:reformulation_drift_p}) and (\ref{eqn:reformulation_drift_n}) hold as well. In addition, if $(X,\mathbb{F})$ is a solution of Eq. (\ref{eqn:SDE_BES}) or Eq. (\ref{eqn:SDE_BES_mvasd}), then $X$ has no occupation time in zero, which is an immediate consequence of the occupation times formula (\ref{eqn:occupationtime}):
\[
	\int_0^t \ind_{\{0\}} (X_s) \, \myd s = \int_0^t \ind_{\{0\}} (X_s) \, \myd \assPro{X}_s 
																				= \int_{\{0\}} L_+^X(t,y) \, \myd y
																				= 0\,, \qquad t < S_\infty^X, \ \mathbf{P}\text{-a.s.} 
\]
Hence, from (\ref{eqn:reformulation_drift_p}) and (\ref{eqn:reformulation_drift_n}) we obtain that any solution $(X,\mathbb{F})$ of Eq. (\ref{eqn:SDE_BES}) or Eq. (\ref{eqn:SDE_BES_mvasd}) satisfies
\[
	\int_0^t \frac{\delta - 1}{2\, X_s} \, \myd s = \int_{\mathbb{R}} L_m^X(t,y) \, \myd f_\delta(y)\,, \qquad t < S_\infty^X, \ \mathbf{P}\text{-a.s.}
\]
Consequently,  Eq. (\ref{eqn:SDE_BES}) coincides with Eq. (\ref{eqn:SDE_BES_mvasd}). Additionally, explosion does not occur. Indeed, for any solution $(X,\mathbb{F})$ of Eq. (\ref{eqn:SDE_BES}) we see that the process $X^2$ satisfies
\[
	X_t^2 = x_0^2 + 2 \int_0^t \sqrt{X^2_s} \, \myd B_s + \delta t\,, \qquad t < S_\infty^{X^2}, \ \mathbf{P}\text{-a.s.},
\]
which is an immediate consequence of the It\^o formula and the relation $S_\infty^X = S_\infty^{X^2}$. Since the coefficients of this equation possess at most linear growth, explosion does not occur, i.e., $S_\infty^X = S_\infty^{X^2} = +\infty$ $\mathbf{P}$-a.s. \\
\indent Using the primitive $G_\delta$ of $1/f_\delta$ and the fact that we have $F = F_- = \{0\}$ for $f_\delta$, Theorem \ref{theorem:isolated_points} states that for any solution $(X,\mathbb{F})$ the transformed process $(Y,\mathbb{F}) = (G_\delta(X),\mathbb{F})$ is a solution to
\begin{equation}\label{eqn:SDE_BES_transformed}
	Y_t = y_0 + \int_0^t \sigma_\delta(Y_s) \, \myd B_s + \int_0^t \ind_{\{0\}}(Y_s) \, \myd Y_s\,,
\end{equation}
where $y_0 = G(x_0)$, $\sigma_\delta = 1/f_\delta \circ H_\delta$ and $H_\delta$ denotes the inverse of $G_\delta$. Additionally, applying Proposition \ref{proposition:drift_part_pointwise_sum}, the drift can be rewritten as
\begin{equation}\label{eqn:jumps_loc_time_BES}
	\int_0^t \ind_{\{0\}}(Y_s) \, \myd Y_s  = \frac{1}{2} \left( L_+^Y(t,0) - L_-^Y(t,0) \right)
																				  = L_m^X(t,0) - L_m^X(t,0-)\,, \qquad t \geq 0, \ \mathbf{P}\text{-a.s.}
\end{equation}
These jumps of the local time $L_m^X$ are a degree of freedom and are responsible for the non-uniqueness of solutions of Eq. (\ref{eqn:SDE_BES_mvasd}). Indeed, in \cite{blei_bessel_2011}, Theorem 2.22, (in connection with \cite{blei_bessel_2011}, Remark 2.26(ii)) it is shown that, for every $\alpha \in (-\infty,\frac{1}{2})$, the equation
\begin{equation}\label{eqn:SDE_BES_controlled}
	\left\{
		\begin{gathered}
			X_t = x_0 + B_t + \int_\mathbb{R} L_m^X(t,y) \, \myd f_{\delta}(y)\,, \medskip \\
			L_m^X(t,0) - L_m^X(t,0-) = 2\, \alpha \, L_m^X(t,0)\,,
		\end{gathered}
	\right.
\end{equation}
possesses a unique solution, the so-called \emph{skew} $\delta$-dimensional Bessel process with skewness parameter $\alpha$ started at $x_0$. For $\alpha = 0$ the solution is, in correspondence with Section \ref{sec:symmetric_solutions}, also called the \emph{symmetric} $\delta$-dimensional Bessel process started at $x_0$. Clearly, for every $\alpha \in (-\infty,\frac{1}{2})$ we obtain solutions of Eq. (\ref{eqn:SDE_BES_mvasd}) different in law although we have $E_b = N_b = \emptyset$ and (\ref{eqn:jumps_loc_time_BES}) becomes
\begin{equation}\label{eqn:BES_part_bounded_var}
	\begin{split}
		\int_0^t \ind_{\{0\}}(Y_s) \, \myd Y_s  = \alpha \, L_+^Y(t,0)
																					  = 2\, \alpha \, L_m^X(t,0), \qquad t \geq 0, \ \mathbf{P}\text{-a.s.}
	\end{split}
\end{equation}

\noindent In addition, we remark that Eq. (\ref{eqn:SDE_BES_controlled}) can be also considered for $\alpha = 1/2$. In that case, for a positive starting point $x_0 \geq 0$ we obtain the $\delta$-dimensional Bessel process as the unique solution of Eq. (\ref{eqn:SDE_BES_controlled}), which stays positive and which is reflected to the positive half line at zero. For a negative starting point $x_0 < 0$ we obtain as well a unique solution that behaves like the Bessel process after it has reached zero, which happens with probability one. But in the following, we exclude the case of reflection. Moreover, as pointed out in \cite{blei_bessel_2011}, Lemma 2.25 and the remarks before, for a parameter $\alpha > 1/2$ there is no solution of Eq. (\ref{eqn:SDE_BES_controlled}). \\
\indent Now we come back to the general equation (\ref{eqn:SDE_mvasd}). Taking a good solution $(X,\mathbb{F})$, the process $(Y,\mathbb{F}) = (G(X),\mathbb{F})$ is a solution of Eq. (\ref{eqn:SDE_transformed_eqn}). If $F$ is countable, via Proposition \ref{proposition:drift_part_pointwise_sum} we obtain additionally
\[
	\int_0^t \ind_{G(F_-)} (Y_s) \, \myd Y_s = \frac{1}{2} \sum_{a \in G(F_-)} \left( L_+^Y(t,a) - L_-^Y(t,a)\right)
																					 = \sum_{a \in F_-} \left( L_m^X(t,a) - L_m^X(t,a-)\right).
\]
As in the example of the Bessel equation (\ref{eqn:SDE_BES_mvasd}), the jumps $L_m^X(t,a) - L_m^X(t,a-)$ of the local time $L_m^X$ in the points of the set $F_-$ are not determined by Eq. (\ref{eqn:SDE_mvasd}). For this reason, we adopt the concept of fixing these jumps similar to (\ref{eqn:SDE_BES_controlled}). In this way, we put more information on the structure of the part $\int_0^t \ind_{G(F_-)}(Y_s) \,\myd Y_s$ of locally bounded variation appearing in Eq. (\ref{eqn:SDE_transformed_eqn}). To realize this idea, we consider a set function $\nu$ defined on the bounded Borel sets such that $\nu$ is a finite signed measure on $\mathscr{B}([-N,N])$ for every $N\in \mathbb{N}$. Moreover, we assume that $\nu$ has no mass on $F_-^c$, i.e.,
\[
	|\nu|(F_-^c \cap [-N,N]) = 0, \qquad N \in \mathbb{N}\,,
\]
where $|\nu|$ denotes the total variation of $\nu$. By means of $\nu$, we control the jump sizes of $L_m^X$ in the points of $F_-$. For this purpose, we consider the equation
\begin{equation}\label{eqn:SDE_mvasd_controlled}
	\left\{
		\begin{aligned}
			\text{(i) \ \ }  & X_t = X_0 + \int_0^t b(X_s) \, \myd B_s + \int_{\mathbb{R}} L_m^X(t,y) \, \myd f(y)\,, \\[1ex]
			\text{(ii) \ \ } & L_m^X(t,a) - L_m^X(t,a-) = 2 \, L_m^X(t,a)\, \nu(\{a\})\,, \qquad a \in F_- \,.
		\end{aligned}
	\right.
\end{equation}
The notion of a solution of Eq. (\ref{eqn:SDE_mvasd_controlled}) is introduced as in Definition \ref{def:solution}, but in addition we require that Eq. (\ref{eqn:SDE_mvasd_controlled})(ii) holds for all $t < S_\infty^X$ $\mathbf{P}$-a.s. Motivated by the observations in the example of the Bessel equation (\ref{eqn:SDE_BES_controlled}), we come to the following 
\begin{defi}
	A solution of Eq. (\ref{eqn:SDE_mvasd_controlled}) is called a \emph{skew solution of Eq. (\ref{eqn:SDE_mvasd}) with skewness parameter $\nu$.}
\end{defi}
\noindent In particular, using $\nu \equiv 0$ in Eq. (\ref{eqn:SDE_mvasd_controlled}), we describe symmetric solutions of Eq. (\ref{eqn:SDE_mvasd}). \\
\indent Denoting by $\nu^G := \nu \circ G^{-1}$ the image of $\nu$ under $G:\mathbb{R}\rightarrow \mathbb{R}$, we concretise our results concerning the space transformations $G$ and $H$ in the new situation of Eq. (\ref{eqn:SDE_mvasd_controlled}).
\begin{proposition}\label{proposition:points_controlled}
	(i) Let $F$ be countable. If $(X,\mathbb{F})$ is a skew good solution of Eq. (\ref{eqn:SDE_mvasd}) with skewness parameter $\nu$, then 
	$(Y,\mathbb{F}) = (G(X),\mathbb{F})$ is a solution of Eq. (\ref{eqn:SDE_transformed_eqn}), which satisfies
	\begin{equation}\label{eqn:drift_as_gen_drift}
		\int_0^t \ind_{G(F_-)}(Y_s) \, \myd Y_s = \int_\mathbb{R} L_+^Y(t,y) \, \nu^G(\myd y), 	\qquad t < S^Y_{G(\mathbb{R})}, \ \mathbf{P}\text{-a.s.},
	\end{equation}
	or, equivalently, 
	\begin{equation}\label{eqn:SDE_transformed_controlled}
		Y_t = Y_0 + \int_0^t \sigma(Y_s)\, \myd B_s + \int_\mathbb{R} L_+^Y(t,y) \, \nu^G(\myd y), 	
		\qquad t < S^Y_{G(\mathbb{R})}, \ \mathbf{P}\text{-a.s.}
	\end{equation}

	(ii) Conversely, for arbitrary $F$ it holds: If $(Y,\mathbb{F})$ is a solution of Eq. (\ref{eqn:SDE_transformed_eqn}) which satisfies 
	(\ref{eqn:drift_as_gen_drift}) or, equivalently, (\ref{eqn:SDE_transformed_controlled}), then $(X,\mathbb{F}) = (H(Y),\mathbb{F})$ is a skew good 
	solution of Eq. (\ref{eqn:SDE_mvasd}) with skewness parameter $\nu$.\smallskip

	(iii) If $F$ is countable, in both statements (i) and (ii), (\ref{eqn:drift_as_gen_drift}) is  as well equivalent to
	\begin{equation}\label{eqn:drift_as_sum}
		L_+^Y(t,a) - L_-^Y(t,a) = 2 \, L_+^Y(t,a) \, \nu^G(\{a\}),  \qquad t < S^Y_{G(\mathbb{R})}, \ a \in G(F_-), \ \mathbf{P}\text{-a.s.}
	\end{equation}

	(iv) Suppose $|F^A \cap [-N,N]| < +\infty$, $N \in \mathbb{N}$. Then any skew solution $(X,\mathbb{F})$ of Eq. (\ref{eqn:SDE_mvasd}) with skewness 
	parameter $\nu$ is also a good solution, and hence $(Y,\mathbb{F}) = (G(X),\mathbb{F})$ is a solution of Eq. 
	(\ref{eqn:SDE_transformed_eqn}) satisfying (\ref{eqn:drift_as_gen_drift}).
\end{proposition}
\begin{proof} For any solution $(Y,\mathbb{F})$ of Eq. (\ref{eqn:SDE_transformed_eqn}), the relations (\ref{eqn:drift_as_gen_drift}) and (\ref{eqn:SDE_transformed_controlled}) are of course equivalent. Moreover, using (\ref{eqn:loc_time_and_variation_process}) and (\ref{eqn:int_wrt_loc_time}), (\ref{eqn:drift_as_gen_drift}) implies
\[
	\begin{split}
		L_+^Y(t,a) - L_-^Y(t,a) &= 2 \int_0^t \ind_{\{a\}}(Y_s) \int_\mathbb{R} L_+^Y(\myd s,y) \, \nu^G(\myd y) \\
														&= 2 \, L_+^Y(t,a) \, \nu^G(\{a\}), \qquad t < S^Y_{G(\mathbb{R})}, \ a \in G(F_-), \ \mathbf{P}\text{-a.s.}
	\end{split}
\]
Conversely, if $F$ is countable, by the same arguments as in (\ref{eqn:drift_part_pointwise_sum}) it is
\[
	\int_0^t \ind_{G(F_-)}(Y_s) \, \myd Y_s = \frac{1}{2} \sum_{a \in G(F_-)} \left(L_+^Y(t,a) - L_-^Y(t,a) \right), 
																						\qquad t < S^Y_{G(\mathbb{R})}, \ \mathbf{P}\text{-a.s.}
\]																			
and (\ref{eqn:drift_as_sum}) implies
\begin{equation}\label{eqn:no_fulfilled_by_singular_part}
	\begin{split}
		\int_0^t \ind_{G(F_-)}(Y_s) \, \myd Y_s	&= \sum_{a \in G(F_-)} L_+^Y(t,a) \, \nu^G(\{a\}) \\
																					  &= \int_\mathbb{R} L_+^Y(t,y) \, \nu^G(\myd y), \qquad t < S^Y_{G(\mathbb{R})}, \ \mathbf{P}\text{-a.s.}
	\end{split}
\end{equation}
which is (\ref{eqn:drift_as_gen_drift}). Hence, we have proven (iii). To show (i), it remains to use Lemma \ref{lemma:relation_loc_times} and (\ref{eqn:SDE_mvasd_controlled})(ii) implies immediately that (\ref{eqn:drift_as_sum}) is fulfilled for $(Y,\mathbb{F}) = (G(X),\mathbb{F})$. To verify (ii), from Theorem \ref{theorem:spacetrans_H} we see at once that Eq. (\ref{eqn:SDE_mvasd_controlled})(i) is satisfied for $X = H(Y)$. Moreover, via Lemma \ref{lemma:relation_loc_times} we conclude
\[
	\begin{split}
		L_m^X(t,a) - L_m^X(t,a-) &= \frac{1}{2} \left( L_+^Y(t,G(a)) - L_-^Y(t,G(a)) \right) \\
														 &= L_+^Y(t,G(a)) \, \nu^G(\{G(a)\})	\\
														 &= 2 \, L_m^X(t,a) \, \nu(\{a\}), \qquad t < S_\infty^X, \ a \in F_-, \ \mathbf{P}\text{-a.s.}
	\end{split}
\]
Finally, to deduce (iv) it just remains to show that the conditions (i) and (ii) of Theorem \ref{theorem:finitely_many_accumulation_points} are fulfilled. If $F_-$ is infinite, then, using  Eq. (\ref{eqn:SDE_mvasd_controlled})(ii), Corollary \ref{corr:comp_interval} and the assumption on $\nu$, we obtain
\[
	\sum_{a \in F_-} \left| L_m^X(t,a) - L_m^X(t,a-) \right| \leq 2 \sum_{a \in F_-} L_m^X(t,a) \, |\nu|(\{a\})	< +\infty, 
	\qquad t < S_\infty^X, \ \mathbf{P}\text{-a.s.}, 
\]
i.e., condition (i) of Theorem \ref{theorem:finitely_many_accumulation_points} is satisfied. Furthermore, let $\{a_1, a_2, \ldots\}$ be an arbitrary enumeration of $F_-$. Then, again by Eq. (\ref{eqn:SDE_mvasd_controlled})(ii), we have
\[
	\left(\sum_{i=1}^n (L_m^X(\,.\,, a_i) - L_m^X(\,.\,, a_i-))\right)_{n\in \mathbb{N}}
		= \int_{\{a_1, \ldots, a_n\}} 2 \, L_m^X(t,y) \, \nu(\myd y)
\] 
which for $n \rightarrow +\infty$ clearly converges $\mathbf{P}$-a.s. locally in variation on $[0, S_\infty^X)$ to 
\[
	\int_{F_-} 2 \, L_m^X(t,y) \, \nu(\myd y) < +\infty\,.
\]
Hence, condition \ref{theorem:finitely_many_accumulation_points}(ii) is also fulfilled, which ends the proof. 
\end{proof}
\indent For uncountable $F$ or, equivalently, for uncountable $F_-$, the equivalence of (\ref{eqn:drift_as_gen_drift}) and (\ref{eqn:drift_as_sum}) does not hold in general. Indeed, in case of an uncountable set $F_-$ it is possible that $\nu$, and hence $\nu^G$, which are concentrated on the Lebesgue null sets $F_-$ and $G(F_-)$, respectively, have a singular continuous part besides a singular discrete part. Thus we cannot justify the last equality in (\ref{eqn:no_fulfilled_by_singular_part}). \\
\indent Proposition \ref{proposition:points_controlled} reveals the relation between solutions of Eq. (\ref{eqn:SDE_mvasd_controlled}), in which we control the jumps of $L_m^X$ in the points of $F_-$ by $\nu$, and solutions of Eq. (\ref{eqn:SDE_transformed_eqn}), which additionally possess a drift of type of Eq. (\ref{eqn:SDE_mvd}). This is the point where the results of H.J. Engelbert and W. Schmidt \cite{engelbert_schmidt:1985} and \cite{engelbert_schmidt:1989_III} concerning equations of type (\ref{eqn:SDE_mvd}) come into play. As in (\ref{eqn:cond_atoms}) and according to \cite{blei_engelbert_2012}, Theorem 2.2 and Corollary 2.5 and 2.6, we always suppose
\[
	\nu(\{a\}) < 1/2, \qquad a \in F_-\,,
\]
which implies $\nu^G(\{a\}) < 1/2$, $a \in G(F_-)$. On the one hand, we want to exclude the phenomenon of reflection in the points of $F_-$, which is described by $\nu(\{a\}) = 1/2$. On the other hand, if we have $\nu(\{a\}) > 1/2$, and therefore $\nu^G(\{a\}) > 1/2$, for an $a \in F_-$, then, in general, the existence of a solution to an equation which has a drift of type (\ref{eqn:SDE_mvd}) fails.\\
\indent Concerning the existence of skew good solutions of Eq. (\ref{eqn:SDE_mvasd}) we can state the following theorem.
\begin{theorem}\label{theorem:ex_skew_sol}
(i) Let $F$ be arbitrary. Suppose $E_{b/\sqrt{f}} \subseteq N_b$ is satisfied. Then for every initial distribution there exists a skew good solution of Eq. (\ref{eqn:SDE_mvasd}) with skewness parameter $\nu$. \smallskip

(ii) Let $F$ be countable. Then for every initial distribution there exists a skew good solution of Eq. (\ref{eqn:SDE_mvasd}) with skewness parameter $\nu$ if and only if the condition $E_{b/\sqrt{f}} \subseteq N_b$ is satisfied.
\end{theorem}
\begin{proof} 
To prove (i) and the sufficiency of the condition $E_{b/\sqrt{f}} \subseteq N_b$ in (ii), we choose an arbitrary initial distribution $\mu$. Lemma \ref{lemma:connection_sets} implies $E_\sigma \subseteq N_\sigma$. Hence, by Theorem \ref{theorem:e_u_gen_drift}(i) the equation 
\[
	Y_t = Y_0 + \int_0^t \sigma(Y_s) \, \myd B_s + \int_\mathbb{R} L_+^Y(t,y) \, \nu^G(\myd y)
\]
possesses a solution $(Y,\mathbb{F})$ for the initial distribution $\mu \circ G^{-1}$. Moreover, since we have $E_\sigma^c \subseteq G(\mathbb{R})$, from Lemma \ref{lemma:stopping} we obtain $Y_t = Y_{t \wedge S^Y_{G(\mathbb{R})}}$, $t \geq 0$, $\mathbf{P}$-a.s. Analogously to (\ref{eqn:drift_int_with_respect_to_process}), we deduce
\[
	\int_0^t \ind_{G(F_-)}(Y_s) \, \myd Y_s = \int_0^t \ind_{G(F_-)}(Y_s) \int_{\mathbb{R}} L_+^Y(\myd s,y) \, \nu^G(\myd y), 
	\qquad t < S_{G(\mathbb{R})}^Y\,, \ \mathbf{P}\text{-a.s.}
\]
Using (\ref{eqn:int_wrt_loc_time}) and the fact that $\nu^G$ is concentrated on $G(F_-)$, it follows
\[
	\int_0^t \ind_{G(F_-)}(Y_s) \, \myd Y_s = \int_\mathbb{R} L_+^Y(t,y) \, \nu^G(\myd y), \qquad t < S_{G(\mathbb{R})}^Y\,, \ \mathbf{P}\text{-a.s.}
\]
Thus, $(Y,\mathbb{F})$ is also a solution of Eq. (\ref{eqn:SDE_transformed_eqn}). Finally, Proposition \ref{proposition:points_controlled}(ii) implies that $(X,\mathbb{F}) =(H(Y),\mathbb{F})$ is a skew good solution of Eq. (\ref{eqn:SDE_mvasd}) with skewness parameter $\nu$ and initial distribution $\mu$. \\
\indent To prove the necessity of $E_{b/\sqrt{f}}\subseteq N_b$ in (ii), we fix $x_0 \in E_{b/\sqrt{f}}$ and take a skew good solution $(X,\mathbb{F})$ of Eq. (\ref{eqn:SDE_mvasd}) with skewness parameter $\nu$ started at $X_0 = x_0$. Setting $Y = G(X)$, by Proposition \ref{proposition:points_controlled}(i) we conclude that $(Y,\mathbb{F})$ is a solution of Eq. (\ref{eqn:SDE_transformed_eqn}) satisfying
\[
	Y_t = Y_0 + \int_0^t \sigma(Y_s)\, \myd B_s + \int_\mathbb{R} L_+^Y(t,y) \, \nu^G(\myd y), 	\qquad t < S^Y_{G(\mathbb{R})}, \ \mathbf{P}\text{-a.s.}
\]
Lemma \ref{lemma:connection_sets} implies $Y_0 = y_0 = G(x_0) \in E_\sigma$. Choosing an interval $(a,b)$ which satisfies $y_0 \in (a,b)$ and $[a,b] \subseteq G(\mathbb{R})$, we define the $\mathbb{F}$-stopping time $S := \inf \{t \geq 0: Y_t \notin (a,b) \}$. For the stopped process $Y^S_t := Y_{S \wedge t}$, $t \geq 0$, we obtain
\begin{equation}\label{eqn:solution_gen_drift_local}
	\begin{split}
		Y_t^S & =	y_0 + \int_0^{S \wedge t} \sigma(Y_s) \, \myd B_s + \int_\mathbb{R} L_+^Y(S \wedge t, y) \, \nu(\myd y) \\
					& =	y_0 + \int_0^t \sigma(Y^S_s) \, \myd B_s + \int_\mathbb{R} L_+^{Y^S}(t, y) \, \nu^G(\myd y), 
									\qquad t \geq 0, \ \mathbf{P}\text{-a.s.},
	\end{split}
\end{equation}
where we used $S \leq S_{G(\mathbb{R})}^Y$ and, in particular, $S < S_{G(\mathbb{R})}^Y$ on $\{S_{G(\mathbb{R})}^Y < +\infty\}$ to write $t \geq 0$ instead of $t < S^Y_{G(\mathbb{R})}$. Moreover, the relation $L_+^Y(S \wedge t, y) = L_+^{Y^S}(t, y)$, $t \geq 0$, $y \in \mathbb{R}$, $\mathbf{P}$-a.s. can be easily deduced from (\ref{eqn:gen_ito_formula}). Introducing $\nu_{(a,b)}^G := \nu^G(\,.\, \cap (a,b))$, due to (\ref{eqn:loc_time_zero_outside_compact_interval}) the drift part of $Y^S$ can be rewritten as
\begin{equation}\label{eqn:solution_gen_drift_local_drift_part}
	\int_\mathbb{R} L_+^{Y^S}(t,y) \, \nu^G(\myd y) = \int_\mathbb{R} L_+^{Y^S}(t,y) \, \nu_{(a,b)}^G(\myd y), \qquad t \geq 0, \ \mathbf{P}\text{-a.s.}
\end{equation}
Summarizing, $(Y^S,\mathbb{F})$ is a solution to an equation of type (\ref{eqn:SDE_mvd}) started at $y_0 \in E_\sigma$. Hence, Lemma \ref{lemma:stopping} implies $Y^S_t = y_0$, $t \geq 0$, $\mathbf{P}$-a.s. and similar to (\ref{eqn:belonging_to_zeros_of_sigma}) and the lines thereafter we conclude $E_{b/\sqrt{f}} \subseteq N_b$.
\end{proof}

For skew good solutions of Eq. (\ref{eqn:SDE_mvasd}) the following uniqueness result holds.
\begin{theorem}\label{theorem:ex_and_un_unique_skew_sol}
(i) Let $F$ be arbitrary. Suppose $E_{b/\sqrt{f}} \subseteq N_b$. If for every initial distribution Eq. (\ref{eqn:SDE_mvasd}) possesses a unique skew good solution with skewness parameter $\nu$, then it holds $E_{b/\sqrt{f}} = N_b$. \smallskip

(ii) Let $F$ be countable. Then for every initial distribution Eq. (\ref{eqn:SDE_mvasd_controlled}) possesses	a unique skew good solution with skewness parameter $\nu$ if and only if the condition $E_{b/\sqrt{f}} = N_b$ is satisfied.
\end{theorem}
\begin{proof} The proof is accomplished similarly to the proof of Theorem \ref{theorem:ex_u_un_symmetric_sol}. To begin with we show (i) and the necessity of $E_{b/\sqrt{f}} = N_b$ in (ii). For this purpose, let us assume that, for every initial distribution, Eq. (\ref{eqn:SDE_mvasd}) possesses a unique skew good solution with skewness parameter $\nu$. In case of a countable set $F$ by means of Theorem \ref{theorem:ex_skew_sol}(ii), this implies $E_{b/\sqrt{f}} \subseteq N_b$. For uncountable $F$ we suppose that $E_{b/\sqrt{f}} \subseteq N_b$ holds. Now we accomplish this part of the proof by contraposition. Let us assume that $E_{b/\sqrt{f}} = N_b$ does not hold, i.e., there exists an $x_0 \in E^c_{b/\sqrt{f}} \cap N_b$. Then, clearly, $\overline{X} \equiv x_0$ is a skew good solution of Eq. (\ref{eqn:SDE_mvasd_controlled}) with skewness parameter $\nu$. On the other hand, we can also find a non-trivial skew good solution of Eq. (\ref{eqn:SDE_mvasd}) with skewness parameter $\nu$ started at $x_0$. Indeed, because of Lemma \ref{lemma:connection_sets}, $E_{b/\sqrt{f}} \subseteq N_b$ implies $E_\sigma \subseteq N_\sigma$. Hence, there exists the so-called fundamental solution $(Y,\mathbb{F})$ to 
\[
	Y_t = Y_0 + \int_0^t \sigma(Y_s) \, \myd B_s + \int_\mathbb{R} L_+^Y(t,y) \, \nu^G(\myd y)
\]
started at $y_0 = G(x_0) \in E_\sigma^c \cap N_\sigma$ and $Y$ is non-trivial (see \cite{engelbert_schmidt:1989_III}, Definition (4.16) and Theorem (4.35)). Moreover, using the same arguments as in the first part of the proof of Theorem \ref{theorem:ex_skew_sol}, it follows that $(Y,\mathbb{F})$ also solves Eq. (\ref{eqn:SDE_transformed_eqn}). Via Proposition \ref{proposition:points_controlled}(ii) we conclude that $(X,\mathbb{F})$ given by $X = H(Y)$ is a non-trivial skew good solution of Eq. (\ref{eqn:SDE_mvasd}) with skewness parameter $\nu$ and initial point $x_0$. \\
\indent To show that the condition $E_{b/\sqrt{f}} = N_b$ is sufficient in (ii), we take two skew good solutions $(X^i,\mathbb{F}^i)$, $i =1,2$, of Eq. (\ref{eqn:SDE_mvasd_controlled}) with skewness parameter $\nu$ defined on $(\Omega^i, \mathcal{F}^i, \mathbf{P}^i)$, $i =1,2$, which possess the same initial distribution. Via Proposition \ref{proposition:points_controlled}(i) we conclude that $(Y^i,\mathbb{F}^i) := (G(X^i), \mathbb{F}^i)$, $i = 1,2$, are two solutions of Eq. (\ref{eqn:SDE_transformed_eqn}) with identical initial distributions which satisfy (\ref{eqn:SDE_transformed_controlled}). As in the proof of Theorem \ref{theorem:ex_u_un_symmetric_sol} we take a sequence $[a_n,b_n]$, $n \in \mathbb{N}$, which satisfies (\ref{eqn:sequence_of_intervals}). Defining $S^i_n := \inf\{t \geq 0: Y^i_t \notin (a_n, b_n)\}$, $n \in \mathbb{N}$, $i =1,2$, and setting $Y^{i,n}_t := Y^i_{t \wedge S^i_n}$, $t \geq 0$, $n \in \mathbb{N}$, similar to (\ref{eqn:solution_gen_drift_local}) and (\ref{eqn:solution_gen_drift_local_drift_part}), it holds
\[
		Y^{i,n}_t = Y^i_0 + \int_0^t \sigma_n(Y^{i,n}_s) \, \myd B^i_s + \int_\mathbb{R} L_+^{Y^{i,n}}(t,y) \, \nu_n^G(\myd y),
				         \qquad t \geq 0, \ \mathbf{P}^i\text{-a.s.}, \ i=1,2,
\]
where $\sigma_n := \sigma \ind_{(a_n,b_n)}$, $\nu^G_n := \nu^G(\,.\, \cap (a_n,b_n))$ and $B^i$, $i =1,2$, is the corresponding Wiener process. That means, $(Y^{i,n},\mathbb{F})$, $i =1,2$, are solutions to the same equation of type (\ref{eqn:SDE_mvd}) with identical initial distribution. The remaining part of the proof is now accomplished as in the proof of Theorem \ref{theorem:ex_u_un_symmetric_sol} after (\ref{eqn:end_of_proof_uniqueness})
\end{proof}

Similarly to Corollary \ref{cor:ex_u_un_isolated_points}, we can give the following Corollary to Theorem \ref{theorem:ex_skew_sol} and \ref{theorem:ex_and_un_unique_skew_sol} and Proposition \ref{proposition:points_controlled}.
\begin{corollary}
	 Suppose $|F^A \cap [-N,N]| < +\infty$, $N \in \mathbb{N}$. Then the following statements hold. \smallskip
	 
	(i) For every initial distribution there exists a skew solution of Eq. (\ref{eqn:SDE_mvasd}) with skewness parameter $\nu$ if and only if the 
	condition $E_{b/\sqrt{f}} \subseteq N_b$ is satisfied. \smallskip
			 
	(ii) For every initial distribution there exists a unique skew solution of Eq. (\ref{eqn:SDE_mvasd}) with skewness parameter $\nu$ if and only if 
	the condition $E_{b/\sqrt{f}} = N_b$ is satisfied. 
\end{corollary}

\section*{Appendix}
\setcounter{equation}{0}
\addtocounter{section}{1}
\renewcommand{\theequation}{A.\arabic{equation}}
\noindent Let $(X,\mathbb{F})$ be a stochastic process with values in $(\overline{\mathbb{R}}, \mathscr{B}(\overline{\mathbb{R}}))$ defined on a probability space $(\Omega, \mathcal{F}, \mathbf{P})$ and let $S$ be an $\mathbb{F}$-stopping time. We call $(X,\mathbb{F})$ a semimartingale up to $S$ if there exists an increasing sequence $(S_n)_{n \in \mathbb{N}}$ of $\mathbb{F}$-stopping times such that $S = \lim_{n\rightarrow +\infty} S_n$ and the process $(X^n,\mathbb{F})$ obtained by stopping $(X,\mathbb{F})$ in $S_n$ is a real-valued semimartingale. Analogously, we introduce the notion of a local martingale up to $S$. We notice that if $S=+\infty$ $\mathbf{P}$-a.s., then any semimartingale up to $S$ is a semimartingale and any local martingale up to $S$ is a local martingale.\\
\indent If $(X,\mathbb{F})$ is a semimartingale up to $S$, then we can find a decomposition
\begin{equation}\label{eqn:semi_decomposition}
	X_t = X_0 + M_t + V_t\,, \qquad t < S, \ \mathbf{P}\text{-a.s.},
\end{equation}
where $(M,\mathbb{F})$ is a local martingale up to $S$ with $M_0 = 0$ and $(V, \mathbb{F})$ is a right-continuous process whose paths are of bounded variation on $[0,t]$ for every $t < S$ and with $V_0=0$. If $X$ is continuous on $[0,S)$, then there exists a decomposition such that $M$ and $V$ are continuous on $[0,S)$ and this decomposition is unique on $[0,S)$. For any continuous local martingale $(M, \mathbb{F})$ up to $S$ by $\assPro{M}$ we denote the continuous increasing process, which is uniquely determined on $[0,S)$, such that $(M^2 - \assPro{M}, \mathbb{F})$ is a continuous local martingale up to $S$ and $\assPro{M}_0 = 0$. For a continuous semimartingale $(X,\mathbb{F})$ up to $S$ we set $\assPro{X} = \assPro{M}$, where $M$ is the continuous local martingale up to $S$ in the decomposition (\ref{eqn:semi_decomposition}) of $X$. \\
\indent We recall some facts which are well-known for continuous semimartingales. See for example \cite{revuzyor}, Ch. VI, \S{}1. Their extension to semimartingales up to a stopping time $S$ is obvious. Let $(X,\mathbb{F})$ be a continuous semimartingale up to the $\mathbb{F}$-stopping time $S$. Then there exists the right local time $L_+^X$ which is a function on $[0,S) \times \mathbb{R}$ into $[0,+\infty)$ such that for every real function $f$ which is the difference of convex functions the generalized It\^o formula holds:
\begin{equation}\label{eqn:gen_ito_formula}
	f(X_t) = f(X_0) + \int_0^t f^\prime_- (X_s) \, \myd X_s + \frac{1}{2} \int_0^t L_+^X(t,y) \, \myd f^\prime_+(y)\,, \qquad t < S, \ \mathbf{P}\text{-a.s.} 
\end{equation}
Thereby, $f^\prime_+$ (resp. $f^\prime_-$) denotes the right (resp. left) derivative of $f$. Moreover, there exists a modification of $L_+^X$ which is increasing and continuous in $t$ as well as in $y$ right-continuous with limits from the left and we always use this modification. \\
\indent By $L^X_-$ we denote the left local time given by
\[
	L_-^X(t,y) = L_+^X(t,y-)\,, \qquad t < S, \, y \in \mathbb{R}\,.
\]
For the local times the so-called \emph{occupation times formula}
\begin{equation}\label{eqn:occupationtime}
	\int_0^t g(X_s) \; \myd \assPro{X}_s = \int_\mathbb{R} L_\pm^X(t,y)\, g(y) \; \myd y\,, \qquad t < S, \ \mathbf{P}\mbox{-a.s.}
\end{equation}
holds true for every locally integrable or non-negative measurable function $g$. Moreover, the local times satisfy
\begin{equation}\label{eqn:int_wrt_loc_time}
	\int_0^t \ind_{\{y\}}(X_s) \, L_\pm^X(\myd s,y) = L_\pm^X(t,y)\,, \qquad t < S, \ y \in \mathbb{R}, \ \mathbf{P}\text{-a.s.},
\end{equation}
\begin{equation}\label{eqn:loc_time_and_variation_process}
	L_+^X(t,y) - L_-^X(t,y) = 2 \int_0^t \ind_{\{y\}}(X_s) \, \myd V_s\,, \qquad t < S, \, y \in \mathbb{R}, \ \mathbf{P}\text{-a.s.} 
\end{equation}
and
\begin{equation}\label{eqn:loc_time_zero_outside_compact_interval}
	L_\pm^X(t,y) = 0, \qquad t < S, \, y \notin \left[\min_{0 \leq s \leq t} X_s, \max_{0 \leq s \leq t} X_s\right], \ \mathbf{P}\text{-a.s.}
\end{equation}
\indent The following lemma describes the relation between the local times of a continuous semimartingale $(X,\mathbb{F})$ and of the transformed semimartingale $(f(X),\mathbb{F})$ in case that $f$ is a semimartingale function with certain properties. This lemma is a slight modification of \cite{assingschmidt}, Lemma I.1.18, and the proof of \cite{assingschmidt} can be easily adapted to the situation of the following
\begin{lemma}\label{lemma:conversion_loc_time}
	Let $(X,\mathbb{F})$ be a continuous semimartingale up to $S$. Furthermore, let $f: \overline{\mathbb{R}}	\rightarrow \overline{\mathbb{R}}$ be a 
	function such that its restriction to the interval $I = (r_1,r_2)$, $-\infty \leq r_1 < r_2 \leq +\infty$, is absolutely continuous and strictly 
	increasing:
	\[
		f(x) = f(\text{$x_0$}) + \int_{x_0}^x f^\prime(y) \, \myd y\,, \qquad x \in I,
	\]
	where $x_0 \in I$ is a fixed point. We assume that in every point $y \in I$ the function $f^\prime$ admits  a limit from the right $f^\prime(y+)$ 
	and from the left $f^\prime(y-)$ in $[0,+\infty]$. Moreover, we suppose that $(f(X), \mathbb{F})$ is a continuous semimartingale up 
	to\footnote{$\inf \emptyset = +\infty$}	$\widetilde{S} := S \wedge \inf\{t < S : X_t \notin I \}$ with the property
	\[
		\assPro{f(X)}_t = \int_0^t (f^\prime(X_s\pm))^2 \, \myd \assPro{X}_s, \qquad t<\widetilde{S},\ \mathbf{P}\text{-a.s.}
	\]
	Denoting $N_\pm := \{x \in I : f^\prime(x\pm) = +\infty\}$, then it holds
	\[
		L_\pm^{f(X)}(t,f(y)) = L_\pm^X(t,y) \, f^\prime(y\pm), 
		\qquad t <\widetilde{S},\ y \in I\setminus N_{\pm}, \ \mathbf{P}\mbox{-a.s.}
	\]
\end{lemma}

The next theorem is a useful convergence result.
\begin{theorem}\label{theorem:convergence}
	Let $(X,\mathbb{F})$ be a continuous semimartingale up to $S$ with decomposition
	\[
		X_t = X_0 + M_t + V_t, \qquad t < S,\ \mathbf{P}\text{-a.s.}
	\]
	Furthermore, let $f_n$, $n \in \mathbb{N}$, $f$ and $h$ be measurable real-valued functions with
	\[
		\lim\limits_{n \rightarrow \infty} f_n = f \qquad \lambda\text{-a.e.}\footnote{$\lambda$ denotes the Lebesgue measure on $\mathbb{R}$.}
	\]
	and
	\[
		|f_n| \leq h \qquad \lambda\text{-a.e.}, \ n \in \mathbb{N}.
	\]
	Moreover, we assume
	\[
		\int_0^t h^2(X_s) \, \myd \assPro{M}_s < +\infty, \qquad t < S, \ \mathbf{P}\text{-a.s.}
	\]
	Then the stochastic integrals of $f_n(X)$, $n \in \mathbb{N}$, and $f(X)$ with respect to $M$ on $[0,S)$ are well-defined and for every $t \geq 0$ 
	we have  
	\[
		\lim\limits_{n \rightarrow \infty}\sup\limits_{0 \leq s \leq t} \left| \int_0^s f_n(X_u) \, \myd M_u - \int_0^s f(X_u) \, \myd M_u \right| = 0 
		\qquad \text{on } \{t < S\}
	\]
	in probability.
\end{theorem}
\begin{proof} We remind of $\assPro{X} = \assPro{M}$. Using the occupation times formula (\ref{eqn:occupationtime}), we get
\[
	\int_0^t f_n^2(X_s) \, \myd \assPro{M}_s \leq \int_0^t h^2(M_s) \, \myd \assPro{X}_s 	< + \infty, 
\]
$t < S, \ \mathbf{P}\text{-a.s.}$ and by observing $|f| \leq h$
\[
	\int_0^t f^2(X_s) \, \myd \assPro{M}_s < + \infty, \qquad t < S, \ \mathbf{P}\text{-a.s.}
\]
This means, as claimed, that the stochastic integrals are well-defined. From our assumptions we now deduce
\[
	\lim\limits_{n \rightarrow \infty} |f_n(y)-f(y)|^2 \, L_+^X(t,y) = 0 \qquad \lambda\text{-a.e.}, \ t < S,
\]
and
\[
	|f_n(y)-f(y)|^2 \, L_+^X(t,y) \leq 4 h^2(y) \, L_+^X(t,y) \qquad \lambda\text{-a.e.}, \ t < S, \ n \in \mathbb{N}.
\]
Furthermore, it holds
\[
	\int_\mathbb{R} 4 h^2(y) \, L_+^X(t,y)\, \myd y = \int_0^t 4 h^2(X_s) \, \myd \assPro{M}_s < +\infty, \qquad t < S, \ \mathbf{P}\text{-a.s.},
\]
where we again applied the occupation times formula (\ref{eqn:occupationtime}). By means of Lebesgue's dominated convergence theorem these observations justify for every $t < S$  $\mathbf{P}\text{-a.s.}$
\begin{equation}\label{eqn:app_1}
	\lim\limits_{n \rightarrow +\infty} \int_0^t \left| f_n(X_s) - f(X_s) \right|^2 \, \myd \assPro{M}_s
			= \lim\limits_{n \rightarrow +\infty}  \int_\mathbb{R} \left| f_n(y) - f(y) \right|^2 \, L_+^X(t,y) \, \myd y 
			= 0\,.
\end{equation}
Let $(S_k)_{k\in \mathbb{N}}$ be an increasing sequence of $\mathbb{F}$-stopping times such that $\lim_{k \rightarrow +\infty} S_k = S$ $\mathbf{P}$-f.s. and for every $k\in \mathbb{N}$ the process $M^k$ obtained by stopping $M$ in $S_k$ is a continuous local martingale. Fixing $t\geq 0$, (\ref{eqn:app_1}) implies for every $k \in \mathbb{N}$
\[
	\lim\limits_{n \rightarrow +\infty} \int_0^t \left| f_n(X_s) - f(X_s) \right|^2 \, \myd \assPro{M^{S_k}}_s 
		=	\lim\limits_{n \rightarrow +\infty} \int_0^{t \wedge S_k} \left| f_n(X_s) - f(X_s) \right|^2 \, \myd \assPro{M}_s = 0
\]
in probability on $\{t < S\}$. Therefore, for arbitrary $\varepsilon > 0$ we conclude
\[\begin{split}
	& \lim_{n \rightarrow +\infty} \mathbf{P} \left( \left\{ \sup\limits_{0 \leq s \leq t} 
																								 \left| \int_0^s f_n(X_u) \, \myd M_u - \int_0^s f(X_u) \, \myd M_u \right| \geq \varepsilon \right\} 
																								 \cap \{t<S\} \right)  \\
	& \phantom{=====} \leq \lim_{n \rightarrow +\infty} \mathbf{P} \left( \left\{ \sup\limits_{0 \leq s \leq t} 
																		\left| \int_0^s f_n(X_u) \, \myd M^{S_k}_u - \int_0^s f(X_u) \, \myd M^{S_k}_u \right| \geq \varepsilon \right\} 
																								 \cap \{t < S_k\} \right) \\
	& \phantom{========================================} + \mathbf{P} \left( \left\{ S_k \leq t < S\right\}\right)\\
	& \phantom{=====} = \mathbf{P} \left( \left\{ S_k \leq t < S\right\}\right)
\end{split}\] 
for all $k \in \mathbb{N}$, where the last equality is justified by \cite{karatzasshreve}, Prop. 3.2.26. Finally, the assertion follows since $\lim_{k \rightarrow +\infty} S_k = S$ $\mathbf{P}$-f.s. 
\end{proof}
 
\bibliographystyle{plain}
\bibliography{references}

\end{document}